\documentclass[a4paper,12pt]{article}
\usepackage{amsmath}
\usepackage{amssymb}
\usepackage{amsthm}
\usepackage{latexsym}
\usepackage{amsmath}
\usepackage{amsthm}
\usepackage[dvipdfmx]{graphicx}
\usepackage{txfonts}
\usepackage{bm}
\usepackage{color}
\usepackage{epic,eepic}

\usepackage{geometry}
\numberwithin{equation}{section}

\newtheorem{theorem}{Theorem}[section]
\newtheorem{proposition}[theorem]{Proposition}

\newtheorem{lemma}[theorem]{Lemma}
\newtheorem{remark}{Remark}[section]
\newtheorem{example}{Example}[section]

\newcommand{\OMIT}[1]{{\bf [OMIT:} #1 \ {\bf --- end OMIT] }}  
   \renewcommand{\OMIT}[1]{}            

\newcommand{\RR}{{\mathbb{R}}}
\newcommand{\ZZ}{{\mathbb{Z}}}
\newcommand{\veczero}{{\bf 0}}
\newcommand{\dom}{{\rm dom\,}}
\newcommand{\unitvec}[1]{e^{#1}}
\newcommand{\unitvecfirst}{e}
\newcommand{\argmin}{{\rm argmin\,}}
\newcommand{\convZ}{\Box\,}

\newcommand{\finbox}{\hspace*{\fill}$\rule{0.2cm}{0.2cm}$}
\newcommand{\todaye}{\the\year/\the\month/\the\day}

\begin{document}

\title{Projection and Convolution Operations for \\
Integrally Convex Functions
\thanks{
This work was supported by The Mitsubishi Foundation, CREST, JST, 
Grant Number JPMJCR14D2, Japan, and
JSPS KAKENHI Grant Numbers 26350430, 26280004, 17K00037.}
}

\author{
Satoko Moriguchi%
\thanks{Department of Economics and Business Administration,
Tokyo Metropolitan University, 
satoko5@tmu.ac.jp}
\ and 
Kazuo Murota%
\thanks{Department of Economics and Business Administration,
Tokyo Metropolitan University, 
murota@tmu.ac.jp}
}

\date{October, 2017; Version \today}

\maketitle

\begin{abstract}
This paper considers projection and convolution operations for 
integrally convex functions,
which constitute a fundamental function class in discrete convex analysis.
It is shown that the class of integrally convex functions
is stable under projection, and this is also the case with
the subclasses of integrally convex functions
satisfying local or global discrete midpoint convexity.
As is known in the literature, 
the convolution of two integrally convex functions
may possibly fail to be integrally convex.
We show that the convolution of an integrally convex function
with a separable convex function remains integrally convex.
We also point out in terms of examples
that the similar statement is false for
integrally convex functions
with local or global discrete midpoint convexity.
\end{abstract}

{\bf Keywords}:
Discrete convex analysis,  Integrally convex function, 
 Minkowski sum,  Infimal convolution, Integer programming

\section{Introduction}
\label{SCintro}

In discrete convex analysis \cite{Mdcasiam,Mbonn09,Mdcaeco16},
a variety of discrete convex functions are considered.
Among others, integrally convex functions,
due to Favati--Tardella \cite{FT90},
constitute a common framework for discrete convex functions.
A function $f: \ZZ^{n} \to \RR \cup \{ +\infty \}$
is called integrally convex if its local convex extension 
$\tilde{f}: \RR^{n} \to \RR \cup \{ +\infty \}$  
is (globally) convex in the ordinary sense, where
$\tilde{f}$ is defined as the collection of convex extensions of $f$ in each 
unit hypercube 
$\{ x =(x_{1}, \ldots,x_{n}) \in \RR\sp{n} 
\mid a_{i} \leq x_{i} \leq a_{i} + 1 \ (i=1,\ldots, n) \}$ 
with $a \in \ZZ^{n}$;
see Section~\ref{SCintcnvfn}.
A proximity theorem for integrally convex functions
has recently been established in \cite{MMTT16proxICissac,MMTT17proxIC},
together with a proximity-scaling algorithm for minimizing integrally convex functions.

A function
$f: \ZZ^{n} \to \RR \cup \{ +\infty \}$
in $x=(x_{1}, \ldots,x_{n}) \in \ZZ^{n}$
is called {\em separable convex} 
if it can be represented as
$f(x) = \varphi_{1}(x_{1}) + \cdots + \varphi_{n}(x_{n})$
with univariate discrete convex functions
$\varphi_{i}: \ZZ \to \RR \cup \{ +\infty \}$
satisfying 
$\varphi_{i}(t-1) + \varphi_{i}(t+1) \geq 2 \varphi_{i}(t)$
for all $t \in \ZZ$.
Separable convex functions are an obvious
 example of integrally convex functions.
Moreover, 
{\rm L}-convex,
${\rm L}^{\natural}$-convex,
{\rm M}-convex,  
${\rm M}^{\natural}$-convex,  
${\rm L}^{\natural}_{2}$-convex, and 
${\rm M}^{\natural}_{2}$-convex functions \cite{Mdcasiam}
and BS-convex and UJ-convex functions \cite{Fuj14bisubmdc}
are integrally convex functions.
An integrally convex function is L$^{\natural}$-convex 
if and only if it is submodular \cite{FM00}.

The concept of integral convexity found applications,
e.g., in economics and game theory.
It is used in formulating discrete fixed point theorems
\cite{Iim10,IMT05,Yan09fixpt}
and designing solution algorithms for discrete systems of nonlinear equations
\cite{LTY11nle,Yan08comp}.
In game theory  integral concavity of payoff functions 
guarantees the existence of a pure strategy equilibrium 
in finite symmetric games \cite{IW14}.

Various operations can be defined for discrete functions  
$f: \mathbb{Z}\sp{n} \to \mathbb{R} \cup \{ +\infty  \}$
through natural adaptations of 
standard operations in convex analysis, such as
\begin{itemize}
\item
origin shift  $f(x) \mapsto f(x+b)$ with an integer vector $b$,

\item
sign inversion of the variable $f(x) \mapsto f(-x)$, 

\item
nonnegative multiplication of function values
$f(x) \mapsto a f(x)$ with $a \geq 0$,

\item
subtraction of a linear function
$f \mapsto  f[-p]$,
where
$f[-p]$ denotes the function defined by
$f[-p](x) = f(x) - \sum_{i=1}\sp{n} p_{i} x_{i}$
for $p \in \RR\sp{n}$.
\end{itemize}
It is known \cite{Mdcasiam,MS01rel}
that these basic operations preserve
integral convexity
as well as
{\rm L}-, 
${\rm L}^{\natural}$-,
{\rm M}-,  
${\rm M}^{\natural}$-,  
${\rm L}^{\natural}_{2}$-, and 
${\rm M}^{\natural}_{2}$-convexity.
For a positive integer $\alpha$,
the $\alpha$-scaling of $f$ means the function  
$f^{\alpha}: \mathbb{Z}\sp{n} \to \mathbb{R} \cup \{ +\infty  \}$
defined by $f^{\alpha}(x) = f(\alpha x) $ for $x \in \mathbb{Z}^{n}$.
{\rm L}- and ${\rm L}^{\natural}$-convexity
are preserved under scaling,
whereas
{\rm M}- and ${\rm M}^{\natural}$-convexity are not 
stable under scaling \cite{Mdcasiam}.
The scaling operation for integrally convex functions is considered 
recently in \cite{MMTT16proxICissac,MMTT17proxIC,MMTT17midpt}.
Integral convexity admits the scaling operation only when $n \leq 2$;
when $n \geq 3$, the scaled function $f^{\alpha}$  is  not necessarily integrally convex.
Within subclasses of integral convex functions
with local or global discrete midpoint convexity,
the scaling operation can be defined for all $n$.

In this paper we are concerned with projection and convolution operations.
For a set $S \subseteq \mathbb{Z}\sp{n+m}$,
the {\em projection} of $S$ (to $\mathbb{Z}\sp{n}$)
is the set $T \subseteq  \mathbb{Z}\sp{n}$ defined by 
\begin{equation}  \label{projsetdef}
 T  = \{ x \in \mathbb{Z}\sp{n} \mid \exists y \in \mathbb{Z}\sp{m}: (x,y) \in S \}. 
\end{equation}
For a function
$f: \mathbb{Z}\sp{n+m} \to \mathbb{R} \cup \{ +\infty  \}$,
the  projection%
\footnote{
In (ordinary) convex analysis, the projection operation 
(\ref{projfndef}) for functions
is also referred to as ``partial minimization''
and the resulting function $g$ as 
``marginal function'' \cite[Def. 2.4.4]{HL01}.
Note that the epigraph of $g$ is the projection of the epigraph
of $f$ \cite[Fig.2.4.1]{HL01}.
} 
 of $f$ to $\mathbb{Z}\sp{n}$ is the function
$g: \mathbb{Z}\sp{n} \to \mathbb{R} \cup \{ +\infty  \}$
defined by 
\begin{equation} \label{projfndef}
 g(x) = \inf \{f(x,y)\mid y\in \mathbb{Z}\sp{m}\}  
\qquad (x \in \mathbb{Z}\sp{n}) , 
\end{equation}
where it is assumed that $g(x) > -\infty$ for all $x$.
For sets $S_{1}$, $S_{2} \subseteq \ZZ\sp{n}$, 
their {\em Minkowski sum} $S_{1}+S_{2}$ is defined by
\begin{equation} \label{minkowsumZdef}
S_{1}+S_{2} = \{ y + z \mid y \in S_{1}, z \in S_{2} \} .
\end{equation}
For functions
$f_{1}, f_{2}: \ZZ\sp{n} \to \RR \cup \{ +\infty \}$,
their (integer infimal) {\em convolution}
is the function
$f_{1} \convZ f_{2}: \ZZ\sp{n} \to \RR \cup \{ +\infty \}$ 
defined by
\begin{equation} \label{f1f2convdef}
(f_{1} \convZ f_{2})(x) =
 \inf\{ f_{1}(y) + f_{2}(z) \mid x= y + z, \,  y, z \in \ZZ\sp{n} \}
\quad (x \in \ZZ\sp{n}),
\end{equation}
where it is assumed that, for every $x \in \ZZ\sp{n}$,
 the infimum on the right-hand side is not equal to $-\infty$.

The following facts are known about projections. 

\begin{itemize}
\item
The projection of a separable convex function
is separable convex.

\item
The projection of an L$^{\natural}$-convex function
is L$^{\natural}$-convex
\cite[Theorem 7.11]{Mdcasiam}.
Similarly, the projection of an L-convex function is L-convex
\cite[Theorem 7.10]{Mdcasiam}.

\item
The projection of an M$^{\natural}$-convex function
is M$^{\natural}$-convex
\cite[Theorem 6.15]{Mdcasiam}.
However, the projection of an M-convex function is 
not necessarily M-convex%
\footnote{
The failure of M-convexity is ascribed to the obvious reason that 
the effective domain of the projected function $g$ does not lie 
on a hyperplane of a constant component sum.
}.
\end{itemize}
The following facts are found in this paper.

\begin{itemize}
\item
The projection of an integrally convex function
is integrally convex
(Theorem~\ref{THicvproj}). 

\item
The projection of an integrally convex function
with global discrete midpoint convexity
is an integrally convex function
with global discrete midpoint convexity
(Theorem~\ref{THsdicvproj}). 

\item
The projection of an integrally convex function
with local discrete midpoint convexity
is an integrally convex function
with local discrete midpoint convexity
(Theorem~\ref{THwdicvproj}). 
\end{itemize}

As for convolutions the following facts are known.

\begin{itemize}
\item
The convolution of separable convex functions is separable convex%
\footnote{
Proof:  If 
$f_{k}(x)  =  \sum_{i=1}\sp{n} \varphi_{ki}(x_{i})$ $(k=1,2)$,
then 
$(f_{1} \convZ f_{2})(x) 
=  \sum_{i=1}\sp{n} (\varphi_{1i} \convZ \varphi_{2i})(x_{i})$.
}.

\item
The convolution of M$^{\natural}$-convex functions 
is M$^{\natural}$-convex
\cite[Theorem 6.15]{Mdcasiam}.
Similarly, the convolution of M-convex functions is M-convex
\cite[Theorem 6.13]{Mdcasiam}.

\item
The convolution of L$^{\natural}$-convex functions 
is not necessarily L$^{\natural}$-convex, 
but is integrally convex \cite[Theorem 8.42]{Mdcasiam}.
Similarly for L-convex functions.

\item
The convolution of an L$^{\natural}$-convex function
and a separable convex function is L$^{\natural}$-convex
\cite[Theorem 7.11]{Mdcasiam}.
Similarly, the convolution of L-convex function
and a separable convex function is L-convex
\cite[Theorem 7.10]{Mdcasiam}.

\item
The convolution of integrally convex functions
is not necessarily integrally convex
\cite[Example 4.12]{MS01rel}, 
\cite[Example 3.15]{Mdcasiam}.
\end{itemize}
The following facts are found in this paper.

\begin{itemize}
\item
The convolution of an integrally convex function
and a separable convex function is integrally convex
(Theorem~\ref{THicsepconvol}).

\item
The convolution of an integrally convex function
with global discrete midpoint convexity
and a separable convex function is 
not necessarily 
an integrally convex function
with global discrete midpoint convexity
(Examples \ref{EXdicdim3set} and \ref{EXdicdim3fn}).

\item
The convolution of an integrally convex function
with local discrete midpoint convexity
and a separable convex function is 
not necessarily 
an integrally convex function
with local discrete midpoint convexity
(Examples \ref{EXdicdim3set} and \ref{EXdicdim3fn}).
\end{itemize}

This paper is organized as follows.
Section~\ref{SCprelim} is a review of relevant results on
integrally convex functions.
Section~\ref{SCproj} deals with projections and
Section~\ref{SCconvol} with convolutions.
Concluding remarks are given in Section~\ref{SCconclrem}.

\section{Preliminaries}
\label{SCprelim}

\subsection{Basic definition and notation}
\label{SCbasicdef}

For integer vectors 
$a \in (\ZZ \cup \{ -\infty \})\sp{n}$ and 
$b \in (\ZZ \cup \{ +\infty \})\sp{n}$ 
with $a \leq b$,
$[a,b]_{\ZZ}$ denotes the integer interval 
(box, discrete rectangle)
between $a$ and $b$,
i.e.,
$[a,b]_{\ZZ} = \{ x \in \ZZ\sp{n} \mid a \leq x \leq b \}$.
For a function $f: \ZZ^{n} \to \RR \cup \{ +\infty \}$,
the effective domain of $f$ means the set
$\{ x \in \ZZ^{n} \mid f(x) <  +\infty \}$ and is denoted by $\dom f$.
The indicator function of a set $S \subseteq \ZZ\sp{n}$
is a function 
$\delta_{S}: \ZZ\sp{n} \to \{ 0, +\infty \}$
defined by
$ \delta_{S}(x)  =
   \left\{  \begin{array}{ll}
    0            &   (x \in S) ,      \\
   + \infty      &   (x \not\in S) . \\
                      \end{array}  \right.$ 
For $x \in \RR^{n}$,
$\lceil x \rceil$ and $\lfloor x \rfloor$ denote
the integer vectors obtained by
componentwise rounding-up and rounding-down to the nearest integers,
respectively.

\subsection{Integrally convex functions}
\label{SCintcnvfn}

For $x \in \RR^{n}$ the integral neighborhood of $x$ is defined as 
\[
N(x) = \{ z \in \mathbb{Z}^{n} \mid | x_{i} - z_{i} | < 1 \ (i=1,\ldots,n)  \}.
\]
In other words, $N(x)= [ \lfloor x \rfloor , \lceil x \rceil ]_{\ZZ}$. 
For a function
$f: \mathbb{Z}^{n} \to \mathbb{R} \cup \{ +\infty  \}$
the local convex extension 
$\tilde{f}: \RR^{n} \to \RR \cup \{ +\infty \}$
of $f$ is defined 
as the union of all convex envelopes of $f$ on $N(x)$.  That is,
\begin{equation} \label{fnconvclosureloc2}
 \tilde f(x) = 
  \min\{ \sum_{y \in N(x)} \lambda_{y} f(y) \mid
      \sum_{y \in N(x)} \lambda_{y} y = x,  
  (\lambda_{y})  \in \Lambda(x) \}
\quad (x \in \RR^{n}) ,
\end{equation} 
where $\Lambda(x)$ denotes the set of coefficients for convex combinations indexed by $N(x)$:
\[ 
  \Lambda(x) = \{ (\lambda_{y} \mid y \in N(x) ) \mid 
      \sum_{y \in N(x)} \lambda_{y} = 1, 
      \lambda_{y} \geq 0 \ (\forall y \in N(x))  \} .
\] 
If $\tilde f$ is convex on $\RR^{n}$,
then $f$ is said to be {\em integrally convex} \cite{FT90,Mdcasiam}.

A set $S \subseteq \ZZ^{n}$ is said to be 
integrally convex if
the convex hull $\overline{S}$ of $S$ coincides with the union of the 
convex hulls of $S \cap N(x)$ over $x \in \RR^{n}$,
i.e., if, for any $x \in \RR^{n}$, 
$x \in \overline{S} $ implies $x \in  \overline{S \cap N(x)}$.
The effective domain of an integrally convex function 
is an integrally convex set.

Integral convexity can be characterized by a local condition
under the assumption that the effective domain is an integrally convex set.

\begin{theorem}[\cite{FT90,MMTT17proxIC}]
\label{THfavtarProp33}
Let $f: \mathbb{Z}^{n} \to \mathbb{R} \cup \{ +\infty  \}$
be a function with an integrally convex effective domain.
Then the following properties are equivalent:

{\rm (a)}
$f$ is integrally convex.

{\rm (b)}
For every $x, y \in \ZZ\sp{n}$ with $\| x - y \|_{\infty} =2$ 
we have \ 
\begin{equation}  \label{intcnvconddist2}
\tilde{f}\, \bigg(\frac{x + y}{2} \bigg) 
\leq \frac{1}{2} (f(x) + f(y)).
\end{equation}
\end{theorem}

Integral convexity of a function
 can also be characterized by integral convexity of the minimizer sets.

\begin{theorem}[{\cite[Theorem 3.29]{Mdcasiam}}]
  \label{THintconvfnsetchar}
Suppose a function $f: \ZZ\sp{n} \to \RR \cup \{ +\infty \}$
has a nonempty bounded effective domain.  Then  $f$ is an integrally convex function 
if and only if 
$\argmin f[-p]$ is an integrally convex set for every $p \in \RR\sp{n}$.
\end{theorem}

\begin{remark} \rm \label{RMintcnvconcept}
The concept of integrally convex functions
is introduced in \cite{FT90} for functions defined on 
integer intervals (discrete rectangles).
The extension to functions with general 
integrally convex effective domains
is straightforward, which is found in \cite{Mdcasiam}.
Theorem~\ref{THfavtarProp33} is proved in \cite[Proposition 3.3]{FT90} 
when the effective domain is an integer interval 
and in the general case in \cite{MMTT17proxIC}. 
\finbox
\end{remark}

\subsection{Discrete midpoint convexity}
\label{SCdiscmidptconv}

A function
$f: \ZZ^{n} \to \RR \cup \{ +\infty \}$
is called
{\em L$^{\natural}$-convex}
\cite{FM00,Mdcasiam}
if it satisfies discrete midpoint convexity
\begin{equation} \label{lnatfmidconv}
 f(x) + f(y) \geq
    f \left(\left\lceil \frac{x+y}{2} \right\rceil\right) 
  + f \left(\left\lfloor \frac{x+y}{2} \right\rfloor\right) 
\end{equation}
for all $x,y \in \ZZ^{n}$.
A function is ${\rm L}\sp{\natural}$-convex
if and only if it is
submodular and integrally convex.
L$^{\natural}$-convex functions
form a well-behaved subclass of integrally convex functions.

A set $S \subseteq  \ZZ\sp{n}$ is called 
L$^{\natural}$-convex if 
\begin{equation} \label{lnatcnvsetdef}
 x, y \in S
\ \Longrightarrow \ 
\left\lceil \frac{x+y}{2} \right\rceil , 
\left\lfloor \frac{x+y}{2} \right\rfloor  \in S. 
\end{equation}
A set $S$ is L$^{\natural}$-convex if and only if
if its indicator function $\delta_{S}$ is 
an L$^{\natural}$-convex function.
The effective domain of an L$^{\natural}$-convex function 
is an L$^{\natural}$-convex set.

A set $S \subseteq  \ZZ\sp{n}$ is called
{\em discrete midpoint convex}
\cite{MMTT17midpt} if 
\begin{equation} \label{dirintcnvsetdef}
 x, y \in S, \ \| x - y \|_{\infty} \geq 2
\ \Longrightarrow \ 
\left\lceil \frac{x+y}{2} \right\rceil , 
\left\lfloor \frac{x+y}{2} \right\rfloor  \in S. 
\end{equation}
Obviously, an ${\rm L}\sp{\natural}$-convex set  
is discrete midpoint convex.
It is also known that a discrete midpoint convex set is integrally convex.

A function $f: \ZZ\sp{n} \to \RR \cup \{ +\infty \}$ 
is called 
{\em  globally discrete midpoint convex}
if the discrete midpoint convexity (\ref{lnatfmidconv})
is satisfied by every pair $(x, y) \in \ZZ\sp{n} \times \ZZ\sp{n}$
with $\| x - y \|_{\infty} \geq 2$.
A function $f: \ZZ\sp{n} \to \RR \cup \{ +\infty \}$ 
is called 
{\em  locally discrete midpoint convex}
if $\dom f$ is a discrete midpoint convex set
and the discrete midpoint convexity (\ref{lnatfmidconv})
is satisfied by every pair $(x, y) \in \ZZ\sp{n} \times \ZZ\sp{n}$
with $\| x - y \|_{\infty} =2$ (exactly equal to two)%
\footnote{
Local discrete midpoint convex functions
are called ``directed integrally convex functions''
in \cite{MMTT16proxICissac}.
}.  
The effective domain 
of a (locally or globally) discrete midpoint convex function
is a discrete midpoint convex set.
A set $S$ is discrete midpoint convex if and only if
if its indicator function $\delta_{S}$ is 
a discrete midpoint convex function.

The inclusion relations 
among the function classes are summarized as follows:
\begin{align}
& \{  \mbox{separable convex functions} \} 
\notag \\
& \subsetneqq \ 
\{ \mbox{${\rm L}\sp{\natural}$-convex functions} \} 
= \{ \mbox{submodular integrally convex functions} \} 
\notag \\
& \subsetneqq \ 
\{ \mbox{globally discrete midpoint convex functions} \} 
\notag \\
& \subsetneqq \ 
\{ \mbox{locally discrete midpoint convex functions} \} 
\notag \\
&  
\subsetneqq \ 
\{ \mbox{integrally convex functions} \}.
  \label{fnclasses3}
\end{align}
All the inclusions above are proper; see \cite{MMTT17midpt}.

An inequality, called
``parallelogram inequality,''
is known for discrete midpoint convex functions.
For any pair of distinct vectors $x, y \in \ZZ\sp{n}$,
we can decompose $y-x$ into $\{ -1,0,+1 \}$-vectors as
\begin{equation} \label{DICdecAkBksum}
y - x = \sum_{k=1}^{m} (\bm{1}_{A_{k}} - \bm{1}_{B_{k}}),
\end{equation}
where $m = \| y - x \|_{\infty}$,
\begin{equation} \label{DICdecAkBkdef}
 A_{k} = \{ i \mid y_{i}- x_{i} \geq m+1-k \},
\quad
 B_{k} = \{ i \mid y_{i}- x_{i} \leq -k \} 
\qquad (k=1,\ldots,m) .
\end{equation}
Note that 
$A_{1} \subseteq A_{2}  \subseteq \cdots \subseteq A_{m}$, \
$B_{1} \supseteq B_{2} \supseteq \cdots \supseteq B_{m}$, \
$A_{m} \cap B_{1} = \emptyset$, and
$A_{1} \cup B_{m} \not= \emptyset$.
The following theorem is a reformulation of
the parallelogram inequality given in \cite{MMTT16proxICissac,MMTT17midpt}.

\begin{theorem}[\cite{MMTT16proxICissac,MMTT17midpt}]  \label{THparallelineqgen-var}
Let $f : \ZZ\sp{n} \to \RR \cup \{+\infty\}$ be 
a (globally or locally) discrete midpoint convex function, 
and $x, y \in \dom f$.
Let
$\displaystyle 
d = \sum_{k \in J} (\bm{1}_{A_{k}} - \bm{1}_{B_{k}})$
for any $J \subseteq \{ 1,2,\ldots, m \}$
in the decomposition {\rm (\ref{DICdecAkBksum})}.
Then
\begin{align}
 f(x) + f(y)  \geq  f(x + d) +  f(y - d) .
\label{paraineqd1d2gen}
\end{align}
\end{theorem}

\section{Projection Operation}
\label{SCproj}

Recall that the projection $T$ of a set 
$S \subseteq \mathbb{Z}\sp{n+m}$
is defined by
\begin{align} 
 T  &= \{ x \in \mathbb{Z}\sp{n} \mid \exists y \in \mathbb{Z}\sp{m}: (x,y) \in S \}, 
 \label{projsetdef-again}
\end{align}
and the projection $g$ of a function 
$f: \mathbb{Z}\sp{n+m} \to \mathbb{R} \cup \{ +\infty  \}$
is defined by
\begin{align} 
 g(x) &= \inf \{f(x,y)\mid y\in \mathbb{Z}\sp{m}\}  
\qquad (x \in \mathbb{Z}\sp{n}) , 
\label{projfndef-again}
\end{align}
where it is assumed that $g(x) > -\infty$ for all $x$.

\subsection{Projection of integrally convex functions}
\label{SCprojIC}

\begin{theorem} \label{THicvsetproj}
The projection of an integrally convex set
is an integrally convex set.
\end{theorem}

\begin{proof}
Let $T \subseteq  \mathbb{Z}\sp{n}$
be the projection of an integrally convex set
$S \subseteq  \mathbb{Z}\sp{n+m}$.
We will show that
$x \in \overline{T}$ implies $x \in  \overline{T \cap N(x)}$.
Let $x \in \overline{T}$.
There exists $y \in \mathbb{R}\sp{m}$ 
such that $(x,y) \in \overline{S}$
(see Lemma~\ref{LMprojconvhull} below).
Then, by integral convexity of $S$,  we have
$(x,y) \in \overline{S \cap N((x,y))}$, i.e., 
there exist 
$(u\sp{(k)},v\sp{(k)}) \in S \cap N((x,y))$ 
and $\lambda_{k}$ $(k=1,2,\ldots,l)$
such that 
\begin{equation*} 
 (x,y) = \sum_{k=1}\sp{l} \lambda_{k} (u\sp{(k)},v\sp{(k)}),
\qquad
 \sum_{k=1}\sp{l} \lambda_{k}  = 1,  \quad   \lambda_{k}  \geq 0.
\end{equation*}
We have
$x = \sum_{k=1}\sp{l} \lambda_{k} u\sp{(k)}$
from the first equation,
and 
$u\sp{(k)} \in T \cap N(x)$
from $(u\sp{(k)},v\sp{(k)}) \in S \cap N((x,y))$.
Hence $x \in  \overline{T \cap N(x)}$. 
\end{proof}

The following fact used in the above proof
is stated and proved for completeness,
though it is just a basic fact about convexity.

\begin{lemma} \label{LMprojconvhull}
$\overline{T} = \{ x \in \mathbb{R}\sp{n} 
  \mid \exists y \in \mathbb{R}\sp{m}: (x,y) \in \overline{S} \}$.
That is, the convex hull of the projection of $S$ 
coincides with the projection of the convex hull of $S$.
\end{lemma}
\begin{proof}
We denote
$ \{ x \in \mathbb{R}\sp{n} 
  \mid \exists y \in \mathbb{R}\sp{m}: (x,y) \in \overline{S} \}$
by ${\rm proj}(\overline{S})$.
To show $\overline{T} \subseteq {\rm proj}(\overline{S})$,
assume $x \in \overline{T}$.
Then there exist
$u\sp{(k)} \in T$ $(k=1,2,\ldots,l)$
such that
\begin{align*} 
 x &= \sum_{k=1}\sp{l} \lambda_{k} u\sp{(k)},
\qquad
 \sum_{k=1}\sp{l} \lambda_{k}  = 1,  \quad   \lambda_{k}  \geq 0 .
\end{align*}
Since $T$ is the projection of $S$
and $u\sp{(k)} \in T$,
there exist $v\sp{(k)}$ such that
$(u\sp{(k)},v\sp{(k)}) \in S$. 
Defining $y = \sum_{k=1}\sp{l} \lambda_{k} v\sp{(k)}$
with the coefficients $\lambda_{1}, \lambda_{2}, \ldots, \lambda_{l}$ above,
we have
$(x,y) = \sum_{k=1}\sp{l} \lambda_{k} (u\sp{(k)},v\sp{(k)})
\in \overline{S}$.
This shows that $x \in {\rm proj}(\overline{S})$.
Hence $\overline{T} \subseteq {\rm proj}(\overline{S})$.

To show the converse
$\overline{T} \supseteq {\rm proj}(\overline{S})$,
assume $x \in {\rm proj}(\overline{S})$.
Then there exists
$y$ such that $(x,y) \in \overline{S}$,
which in turn implies
$(x,y) = \sum_{k=1}\sp{l} \lambda_{k} (u\sp{(k)},v\sp{(k)})$
for some
$(u\sp{(k)},v\sp{(k)}) \in S$ and $\lambda_{k}$
$(k=1,2,\ldots,l)$.
Therefore,
$x = \sum_{k=1}\sp{l} \lambda_{k} u\sp{(k)}$
and $u\sp{(k)} \in T$
$(k=1,2,\ldots,l)$,
which shows $x \in \overline{T}$.
\end{proof}

\begin{theorem} \label{THicvproj}
The projection of an integrally convex function
is an integrally convex function.
\end{theorem}

\begin{proof}
Let $g$ be the projection of an integrally convex function $f$.
The effective domain $\dom g$ of $g$ coincides with the projection of $\dom f$,
whereas $\dom f$ is an integrally convex set
by the integral convexity of $f$.
By Theorem~\ref{THicvsetproj}, $\dom g$ is an integrally convex set.
Then by Theorem~\ref{THfavtarProp33}
it suffices to show
\begin{equation} \label{eq1-v2}
  \frac{1}{2} \ [g(x\sp{(1)}) + g(x\sp{(2)}) ] 
\geq   \tilde g \left( \frac{x\sp{(1)}+x\sp{(2)}}{2} \right) 
\end{equation}
 for any 
$x\sp{(1)}, x\sp{(2)} \in \dom g$
with
$\| x\sp{(1)} - x\sp{(2)} \|_{\infty} = 2$,
where $\tilde g$ is the local convex extension of $g$.
Take any  $\varepsilon > 0$.
By the definition (\ref{projfndef-again}) of projection,
there exist
$y\sp{(1)}, y\sp{(2)} \in \ZZ\sp{m}$
such that
$g(x\sp{(i)}) \geq  f(x\sp{(i)}, y\sp{(i)}) - \varepsilon$  for $i=1,2$,
which implies
\begin{equation}  \label{icprf8} 
\frac{1}{2} \ [ g(x\sp{(1)}) + g(x\sp{(2)}) ] \geq  
 \frac{1}{2} \ [f(x\sp{(1)}, y\sp{(1)}) + f(x\sp{(2)}, y\sp{(2)})] -  \varepsilon. 
\end{equation}
Consider the local convex extension $\tilde f(z)$
at $z = [(x\sp{(1)},y\sp{(1)})+(x\sp{(2)},y\sp{(2)})]/2 \in \mathbb{R}\sp{n+m}$.
There exist
$(u\sp{(k)},v\sp{(k)}) \in N(z)$ 
and $\lambda_{k}$ $(k=1,2,\ldots,l)$
such that 
\begin{align} 
& z = \sum_{k=1}\sp{l} \lambda_{k} (u\sp{(k)},v\sp{(k)}),
\quad
 \tilde f(z) = \sum_{k=1}\sp{l} \lambda_{k} f(u\sp{(k)},v\sp{(k)}),
\quad
 \sum_{k=1}\sp{l} \lambda_{k}  = 1,  \quad   \lambda_{k}  \geq 0 .
\label{eq2point-v2}
\end{align}
By Theorem~\ref{THfavtarProp33} for $f$ and the definition of projection $g$
we have
\begin{align}
& \frac{1}{2} \ [f(x\sp{(1)}, y\sp{(1)}) + f(x\sp{(2)}, y\sp{(2)})]
\geq \tilde f(z) = \sum_{k=1}\sp{l} \lambda_{k} f(u\sp{(k)},v\sp{(k)})
\geq 
 \sum_{k=1}\sp{l} \lambda_{k} g(u\sp{(k)}).
 \label{eq2z-v9}
\end{align}
Furthermore, we have
\begin{align} 
 \sum_{k=1}\sp{l} \lambda_{k} g(u\sp{(k)}) 
 \geq \tilde g\left( \frac{x\sp{(1)}+x\sp{(2)}}{2} \right),
\label{eq7-v2}
\end{align}
since
$(x\sp{(1)}+x\sp{(2)})/2 = \sum_{k=1}\sp{l} \lambda_{k} u\sp{(k)}$
by (\ref{eq2point-v2}) and
$u\sp{(k)} \in N( (x\sp{(1)}+x\sp{(2)})/2 )$.
It follows from (\ref{icprf8}), (\ref{eq2z-v9}),  and (\ref{eq7-v2}) that
\begin{equation*}  
\frac{1}{2} \ [ g(x\sp{(1)}) + g(x\sp{(2)}) ] 
 \geq \tilde g\left( \frac{x\sp{(1)}+x\sp{(2)}}{2} \right)
 -  \varepsilon.
\end{equation*}
This implies (\ref{eq1-v2}), since $\varepsilon > 0$ is arbitrary.
\end{proof}

\begin{remark} \rm 
In (ordinary) convex analysis,
convexity of functions is characterized by
convexity of epigraphs.
This characterization makes it possible to reduce 
the proof of convexity for projected functions (marginal functions)
to that for projected sets.
In discrete convex analysis, however,
convexity concepts for functions
such as integral convexity,
${\rm L}^{\natural}$-convexity, and
${\rm M}^{\natural}$-convexity,
do not admit simple characterizations
in terms of the corresponding discrete convexity of epigraphs.
Thus we need separate proofs for sets and functions.
\finbox
\end{remark}

\begin{remark} \rm \label{RMuseproj}
Suppose that $f(x,y)$ is integrally convex in $(x,y)$ and
${\rm L}^{\natural}$- or ${\rm M}^{\natural}$-convex in $y$,
and that $\dom f$ is bounded.
We can minimize such a function efficiently 
on the basis of Theorem \ref{THicvproj} 
if the dimension of $x$ is small%
\footnote{
This fact is pointed out by Fabio Tardella.
}.  
We denote by $n_{x}$ and $n_{y}$ the dimensions of $x$ and $y$, respectively,
and by $K_{x}$ and $K_{y}$ the $\ell_{\infty}$-sizes of $\dom f$ 
projected on the spaces of $x$ and $y$, respectively.
The minimization of $f$ can be formulated 
as the minimization of the projected function $g(x)$ 
defined by (\ref{projfndef}).
Since $g$ is integrally convex by Theorem \ref{THicvproj},
the algorithm of \cite{MMTT17proxIC}
can be used to find a minimum of $g$ with $C(n_{x}) \log_{2} K_{x}$ evaluations of $g$,
where $C(n_{x})$ is superexponential in $n_{x}$.
The evaluation of $g(x)$ itself amounts to minimizing 
$f(x,y)$ over $y$, which can be done 
in polynomial time in $n_{y}$ and $\log_{2} K_{y}$
using ${\rm L}^{\natural}$- or ${\rm M}^{\natural}$-convex 
function minimization algorithms \cite{Mdcasiam}.
Concerning the above-mentioned conditions on $f(x,y)$,
the following facts are known for a quadratic function
$f(x,y)$ represented as
$f(x,y) = (x,y) 
{\small\scriptsize
\left[ \begin{array}{cc}
Q_{xx} & Q_{xy}   \\
Q_{yx} & Q_{yy}   
\end{array}\right]  }
{\small\footnotesize
\left(\begin{array}{c}
x \\ y   
\end{array}\right)  }
+ (c_{x}, c_{y}) 
{\small\footnotesize
\left(\begin{array}{c}
x \\ y   
\end{array}\right) }
$
with a symmetric matrix
$ Q  = {\small\scriptsize
\left[ \begin{array}{cc}
Q_{xx} & Q_{xy}   \\
Q_{yx} & Q_{yy}   
\end{array}\right] }
$.
Such $f$ is integrally convex in $(x,y)$ if 
$Q$ is diagonally dominant with nonnegative diagonals 
(i.e., $q_{ii} \geq \sum_{j \not= i} |q_{ij}|$ for all $i$)
\cite{FT90};
$f$ is ${\rm L}^{\natural}$-convex in $y$
if and only if 
$Q_{yy} = (q\sp{yy}_{ij} )$  is diagonally dominant with nonnegative diagonals
(i.e., $q\sp{yy}_{ii} \geq \sum_{j \not= i} |q\sp{yy}_{ij}|$ for all $i$)
and 
$q\sp{yy}_{ij}  \leq 0 $ for all $i \not= j$ 
\cite{Mdcasiam};
and 
$f$ is ${\rm M}^{\natural}$-convex in $y$
if and only if 
$Q_{yy}$ satisfies 
$q\sp{yy}_{ij}  \geq 0 $ for all $(i,j)$ and
$q\sp{yy}_{ij}  \geq \min ( q\sp{yy}_{ik} ,  q\sp{yy}_{jk} )$
if 
$\{ i,j \} \cap \{ k \} = \emptyset $ 
\cite{Mbonn09}.
\finbox
\end{remark}

\subsection{Projection of discrete midpoint convex functions}
\label{SCprojdirIC}

We begin with sets.

\begin{theorem} \label{THdicvsetproj}
The projection of a discrete midpoint convex set
is a discrete midpoint convex set.
\end{theorem}
\begin{proof}
Let $T \subseteq  \mathbb{Z}\sp{n}$
be the projection (\ref{projsetdef-again}) 
of a discrete midpoint convex set
$S \subseteq  \mathbb{Z}\sp{n+m}$.
To show discrete midpoint convexity (\ref{dirintcnvsetdef}) for $T$,
take $x\sp{(1)}, x\sp{(2)} \in T$ with 
$\| x\sp{(1)} - x\sp{(2)} \|_{\infty} \geq 2$.
By the definition of projection,
we have
$(x\sp{(1)}, y\sp{(1)}) \in S$ and
$(x\sp{(2)}, y\sp{(2)}) \in S$
for some
$y\sp{(1)}, y\sp{(2)} \in \ZZ\sp{m}$.
Since
$\| (x\sp{(1)},y\sp{(1)}) - (x\sp{(2)},y\sp{(2)}) \|_{\infty} \geq 2$,
discrete midpoint convexity (\ref{dirintcnvsetdef}) for $S$
shows
\begin{equation*}  
\left\lceil 
          \frac{(x\sp{(1)},y\sp{(1)})+(x\sp{(2)},y\sp{(2)})}{2} 
    \right\rceil
\in S, 
\qquad
 \left\lfloor 
          \frac{(x\sp{(1)},y\sp{(1)})+(x\sp{(2)},y\sp{(2)})}{2} 
     \right\rfloor
\in S,
\end{equation*}
in which
\begin{align}
 \left\lceil     \frac{(x\sp{(1)},y\sp{(1)})+(x\sp{(2)},y\sp{(2)})}{2}    \right\rceil 
&=
\left( \ 
 \left\lceil    \frac{x\sp{(1)}+x\sp{(2)}}{2}  \right\rceil ,
 \left\lceil    \frac{y\sp{(1)}+y\sp{(2)}}{2}  \right\rceil 
\ \right) ,
\label{projup}
\\
 \left\lfloor     \frac{(x\sp{(1)},y\sp{(1)})+(x\sp{(2)},y\sp{(2)})}{2}    \right\rfloor 
&=
\left( \ 
 \left\lfloor    \frac{x\sp{(1)}+x\sp{(2)}}{2}  \right\rfloor ,
 \left\lfloor    \frac{y\sp{(1)}+y\sp{(2)}}{2}  \right\rfloor 
\ \right). 
\label{projdown}
\end{align}
Therefore,
$\left\lceil 
          \frac{x\sp{(1)}+x\sp{(2)}}{2} 
    \right\rceil \in T$
and
$\left\lfloor 
          \frac{x\sp{(1)}+x\sp{(2)}}{2} 
     \right\rfloor \in T$.
Hence (\ref{dirintcnvsetdef}) holds for $T$.
\end{proof}

For functions we have the following theorems,
the first for the global version 
of discrete midpoint convex functions,
and the second for the local version. 
The proof for the local version relies on the 
parallelogram inequality in Theorem~\ref{THparallelineqgen-var}.

\begin{theorem} \label{THsdicvproj}
The projection of a globally discrete midpoint convex function
is a globally discrete midpoint convex function.
\end{theorem}
\begin{proof}
Let $g$ be the projection  (\ref{projfndef-again})
of a globally discrete midpoint convex function $f$.
To show discrete midpoint convexity (\ref{lnatfmidconv}) for $g$, 
take $x\sp{(1)}, x\sp{(2)} \in \dom g$
with
$\| x\sp{(1)} - x\sp{(2)} \|_{\infty} \geq 2$
and any  $\varepsilon > 0$.
By the definition of projection,
there exist
$y\sp{(1)}, y\sp{(2)} \in \ZZ\sp{m}$
such that
$g(x\sp{(i)}) \geq  f(x\sp{(i)}, y\sp{(i)}) - \varepsilon$  for $i=1,2$,
which implies
\begin{equation}  \label{strprf8} 
 g(x\sp{(1)}) + g(x\sp{(2)}) \geq  
f(x\sp{(1)}, y\sp{(1)}) + f(x\sp{(2)}, y\sp{(2)}) - 2 \varepsilon. 
\end{equation}
Noting
$\| (x\sp{(1)},y\sp{(1)}) - (x\sp{(2)},y\sp{(2)}) \|_{\infty} \geq 2$,
we use discrete midpoint convexity (\ref{lnatfmidconv}) of $f$,
as well as (\ref{projup}) and (\ref{projdown}), to obtain
\begin{align}
& f(x\sp{(1)},y\sp{(1)}) + f(x\sp{(2)},y\sp{(2)}) 
\notag \\
& \geq
    f \left(\left\lceil 
          \frac{(x\sp{(1)},y\sp{(1)})+(x\sp{(2)},y\sp{(2)})}{2} 
    \right\rceil\right) 
  + f \left(\left\lfloor 
          \frac{(x\sp{(1)},y\sp{(1)})+(x\sp{(2)},y\sp{(2)})}{2} 
     \right\rfloor\right) 
\notag \\
 & \geq 
    g \left(\left\lceil 
          \frac{x\sp{(1)}+x\sp{(2)}}{2} 
    \right\rceil\right) 
  + g \left(\left\lfloor 
          \frac{x\sp{(1)}+x\sp{(2)}}{2} 
     \right\rfloor\right) . 
  \label{strprf31112} 
\end{align}
The combination of (\ref{strprf8}) and (\ref{strprf31112}) yields
\begin{equation*}  
 g(x\sp{(1)}) + g(x\sp{(2)})
 \geq 
    g \left(\left\lceil 
          \frac{x\sp{(1)}+x\sp{(2)}}{2} 
    \right\rceil\right) 
  + g \left(\left\lfloor 
          \frac{x\sp{(1)}+x\sp{(2)}}{2} 
     \right\rfloor\right) 
 - 2 \varepsilon .  
\end{equation*}
This implies (\ref{lnatfmidconv}) for $g$, 
since $\varepsilon > 0$ is arbitrary.
\end{proof}

\begin{theorem} \label{THwdicvproj}
The projection of a locally discrete midpoint convex function
is a locally discrete midpoint convex function.
\end{theorem}

\begin{proof}
Let $g(x)$ be the projection (\ref{projfndef-again})
of a locally discrete midpoint convex function $f(x,y)$,
where $y$ is $m$-dimensional.
We may assume $m=1$, since a one-dimensional projection 
repeated $m$ times amounts to an $m$-dimensional projection.
First, $\dom g$ is a discrete midpoint convex set
by Theorem~\ref{THdicvsetproj}.

To show discrete midpoint convexity (\ref{lnatfmidconv}) for $g$, 
take $x\sp{(1)}, x\sp{(2)} \in \dom g$
with
$\| x\sp{(1)} - x\sp{(2)} \|_{\infty} = 2$
and any  $\varepsilon > 0$.
By the definition of projection,
there exist
$y\sp{(1)}, y\sp{(2)} \in \ZZ$
such that
$g(x\sp{(i)}) \geq  f(x\sp{(i)}, y\sp{(i)}) - \varepsilon$  for $i=1,2$,
which implies
\begin{equation}  \label{strprf8w} 
 g(x\sp{(1)}) + g(x\sp{(2)}) \geq  
f(x\sp{(1)}, y\sp{(1)}) + f(x\sp{(2)}, y\sp{(2)}) - 2 \varepsilon .
\end{equation}
Take $z\sp{(1)}, z\sp{(2)} \in \ZZ$
that minimize
$|z\sp{(1)} - z\sp{(2)}|$
subject to
\begin{equation}  \label{strprf9w} 
 g(x\sp{(1)}) + g(x\sp{(2)}) \geq  
f(x\sp{(1)}, z\sp{(1)}) + f(x\sp{(2)}, z\sp{(2)}) - 2 \varepsilon .
\end{equation}
Such $(z\sp{(1)}, z\sp{(2)} )$ exists by (\ref{strprf8w}).
We may assume $z\sp{(2)} - z\sp{(1)} \geq 0$
by interchanging $(x\sp{(1)},z\sp{(1)})$ and 
$(x\sp{(2)},z\sp{(2)})$ if necessary.

Consider the decomposition of 
$(x\sp{(2)},z\sp{(2)}) - (x\sp{(1)},z\sp{(1)})$
into vectors of $\{ -1,0,+1 \}\sp{n+1}$ 
as in (\ref{DICdecAkBksum}):
\begin{equation}  \label{strprf1A1B} 
(x\sp{(2)},z\sp{(2)}) - (x\sp{(1)},z\sp{(1)})
=\sum_{k=1}^{m} (\bm{1}_{A_{k}} - \bm{1}_{B_{k}}),
\end{equation}
where
$m =\| (x\sp{(1)},z\sp{(1)}) - (x\sp{(2)},z\sp{(2)}) \|_{\infty} 
= \max(2 , z\sp{(2)} - z\sp{(1)})$.
It should be clear that the components of the vector
$(x\sp{(i)},z\sp{(i)})$
are numbered by $1,2,\ldots,n$ and $n+1$,
and accordingly, $A_{k}, B_{k} \subseteq \{ 1,2,\ldots,n, n+1 \}$.

\medskip

Claim: \  $z\sp{(2)} - z\sp{(1)} \leq 4$.
\\
(Proof)
To prove the claim by contradiction, suppose that
 $z\sp{(2)} - z\sp{(1)} \geq 5$.
Then we have $m \geq 5$,
$A_{k}=\{ n+1 \}$ for $1 \leq k \leq m-2$, and
$B_{k}=\emptyset$ for $3 \leq k \leq m$.
Hence
$A_{3}=\{ n+1 \}$ and $B_{3}=\emptyset$
since $m \geq 5$.
By parallelogram inequality (\ref{paraineqd1d2gen})
for $f$ with
$x = (x\sp{(1)},z\sp{(1)})$,
$y = (x\sp{(2)},z\sp{(2)})$, and
$d = \bm{1}_{A_{3}} - \bm{1}_{B_{3}} = ( \bm{0}, 1)$,
we obtain
\begin{equation}  \label{strprf10w} 
f(x\sp{(1)}, z\sp{(1)}) + f(x\sp{(2)}, z\sp{(2)}) \geq 
f(x\sp{(1)}, z\sp{(1)}+1) + f(x\sp{(2)}, z\sp{(2)}-1) .
\end{equation}
This is a contradiction to the choice of
$(z\sp{(1)},z\sp{(2)})$, since
$(z\sp{(1)}+1, z\sp{(2)}-1)$ satisfies (\ref{strprf9w}) by (\ref{strprf10w})
and
$|(z\sp{(1)}+1) - (z\sp{(2)}-1)| = |z\sp{(1)} - z\sp{(2)}|-2$. 
Thus the claim is proved.

\medskip

We consider three cases, according to the value of 
$z\sp{(2)} - z\sp{(1)} \in \{ 0,1,2,3, 4 \}$.

{\bf Case 1} ($0 \leq z\sp{(2)} - z\sp{(1)} \leq 2$):
In this case we have 
$\| (x\sp{(1)},z\sp{(1)}) - (x\sp{(2)},z\sp{(2)}) \|_{\infty} = 2$,
which allows us to use
discrete midpoint convexity (\ref{lnatfmidconv})
to obtain
\begin{align} 
& f(x\sp{(1)},z\sp{(1)}) + f(x\sp{(2)},z\sp{(2)}) 
\notag \\
& \geq
    f \left(\left\lceil 
          \frac{(x\sp{(1)},z\sp{(1)})+(x\sp{(2)},z\sp{(2)})}{2} 
    \right\rceil\right) 
  + f \left(\left\lfloor 
          \frac{(x\sp{(1)},z\sp{(1)})+(x\sp{(2)},z\sp{(2)})}{2} 
     \right\rfloor\right) 
\notag \\
 & \geq 
    g \left(\left\lceil 
          \frac{x\sp{(1)}+x\sp{(2)}}{2} 
    \right\rceil\right) 
  + g \left(\left\lfloor 
          \frac{x\sp{(1)}+x\sp{(2)}}{2} 
     \right\rfloor\right) ,
 \label{strprf31112w} 
\end{align}
where the first inequality is discrete midpoint convexity,
and the second inequality is by the definition of
projection as well as 
the identities (\ref{projup}) and (\ref{projdown}) 
with $y\sp{(i)}$ replaced by $z\sp{(i)}$.
Then it follows from 
(\ref{strprf9w}) and (\ref{strprf31112w}) that
\begin{equation}  \label{strprf5w} 
 g(x\sp{(1)}) + g(x\sp{(2)})
 \geq 
    g \left(\left\lceil 
          \frac{x\sp{(1)}+x\sp{(2)}}{2} 
    \right\rceil\right) 
  + g \left(\left\lfloor 
          \frac{x\sp{(1)}+x\sp{(2)}}{2} 
     \right\rfloor\right) 
 - 2 \varepsilon  .
\end{equation}
This implies discrete midpoint convexity (\ref{lnatfmidconv}) for $g$, 
since $\varepsilon > 0$ is arbitrary.

{\bf Case 2} ($z\sp{(2)} - z\sp{(1)} = 3$):
In this case we have $m=3$.
With the notation
$X(p) = 
\{ i \mid 1 \leq i \leq n, \  x\sp{(2)}_{i} - x\sp{(1)}_{i} = p \}$
for $p \in \{ -2,-1, 1, 2 \}$, 
we have
\[
\begin{array}{lll}
A_{1}= \{ n+1 \}, & \quad &B_{1}= X(-1) \cup X(-2),
\\
A_{2}= \{ n+1 \} \cup X(1),& &B_{2}= X(-2),
\\
A_{3}= \{ n+1 \} \cup X(1) \cup X(2), & & B_{3}= \emptyset ,
\end{array}
\]
where
$X(2) \cup X(-2) \not= \emptyset $
by $\| x\sp{(1)} - x\sp{(2)} \|_{\infty} = 2$.
Define
$d = \bm{1}_{A_{1}} - \bm{1}_{B_{1}}$
if $X(2) \not= \emptyset $; 
and
$d = \bm{1}_{A_{3}} - \bm{1}_{B_{3}}$
otherwise.
We denote $d=(b,1)$ with $b \in \ZZ\sp{n}$.
By parallelogram inequality (\ref{paraineqd1d2gen})
for $f$ with
$x = (x\sp{(1)},z\sp{(1)})$,
$y = (x\sp{(2)},z\sp{(2)})$, and $d$ above,
we obtain
\begin{equation}  \label{strprf10w3} 
f(x\sp{(1)}, z\sp{(1)}) + f(x\sp{(2)}, z\sp{(2)}) \geq 
f(x\sp{(1)}+b, z\sp{(1)}+1) + f(x\sp{(2)}-b, z\sp{(2)}-1) .
\end{equation}
Here we have
$\| (x\sp{(1)}+b, z\sp{(1)}+1) 
- (x\sp{(2)}-b, z\sp{(2)}-1) \|_{\infty} =2$,
since
$| (z\sp{(1)}+1) - (z\sp{(2)}-1) | =1$
and
$| (x\sp{(1)}+b)_{i} - (x\sp{(2)}-b)_{i} | \leq 2$
with equality for some $i \in X(2) \cup X(-2)$
by the choice of $d$.
This allows us to use
discrete midpoint convexity (\ref{lnatfmidconv})
to obtain
\begin{equation}  \label{strprf11w3} 
\mbox{RHS of (\ref{strprf10w3})} 
 \geq
    f \left(\left\lceil 
          \frac{(x\sp{(1)},z\sp{(1)})+(x\sp{(2)},z\sp{(2)})}{2} 
    \right\rceil\right) 
  + f \left(\left\lfloor 
          \frac{(x\sp{(1)},z\sp{(1)})+(x\sp{(2)},z\sp{(2)})}{2} 
     \right\rfloor\right) .
\end{equation}
The combination of (\ref{strprf10w3}) and 
(\ref{strprf11w3})
yields
(\ref{strprf31112w}).
The rest of the proof is the same as in Case 1.

{\bf Case 3} ($z\sp{(2)} - z\sp{(1)} = 4$):
In this case we have $m=4$.
With the notation $X(p)$ introduced in Case 2
we have
\[
\begin{array}{lll}
A_{1}= \{ n+1 \}, & \quad &B_{1}= X(-1) \cup X(-2),
\\
A_{2}= \{ n+1 \},& &B_{2}= X(-2),
\\
A_{3}= \{ n+1 \} \cup X(1) , & & B_{3}= \emptyset ,
\\
A_{4}= \{ n+1 \} \cup X(1) \cup X(2), & & B_{4}= \emptyset .
\end{array}
\]
Define
$d = \bm{1}_{A_{1}} - \bm{1}_{B_{1}}$
and denote $d=(b,1)$ with $b \in \ZZ\sp{n}$.
By parallelogram inequality (\ref{paraineqd1d2gen})
for $f$ with
$x = (x\sp{(1)},z\sp{(1)})$,
$y = (x\sp{(2)},z\sp{(2)})$, and $d$ above,
we obtain
(\ref{strprf10w3}),
in which
$\| (x\sp{(1)}+b, z\sp{(1)}+1) 
- (x\sp{(2)}-b, z\sp{(2)}-1) \|_{\infty} =2$
since $| (z\sp{(1)}+1) - (z\sp{(2)}-1) | =2$.
The rest of the proof is the same as in Case 2.
\end{proof}

\begin{remark} \rm 
Theorem~\ref{THdicvsetproj} for discrete midpoint convex sets
can be proved as a special case of Theorem~\ref{THsdicvproj} 
for globally discrete midpoint convex functions,
since a set $S$ is discrete midpoint convex
if and only if its indicator function $\delta_{S}$ 
is globally discrete midpoint convex.
In this paper, however, we have given a direct proof
to Theorem~\ref{THdicvsetproj},
which is shorter and more transparent.
It is emphasized that Theorem~\ref{THicvsetproj} 
for integrally convex sets cannot be proved as a special case of 
Theorem~\ref{THicvproj} for integrally convex functions,
since the proof of the latter depends on the former.
\finbox
\end{remark}

\section{Convolution Operation}
\label{SCconvol}

Recall that the Minkowski sum 
of $S_{1}$, $S_{2} \subseteq \ZZ\sp{n}$ 
is defined by
\begin{equation} \label{minkowsumZdef-again}
S_{1}+S_{2} = \{ y + z \mid y \in S_{1}, z \in S_{2} \} ,
\end{equation}
and the convolution of
$f_{1}, f_{2}: \ZZ\sp{n} \to \RR \cup \{ +\infty \}$
is defined by
\begin{equation} \label{f1f2convdef-again}
(f_{1} \convZ f_{2})(x) =
 \inf\{ f_{1}(y) + f_{2}(z) \mid x= y + z, \  y, z \in \ZZ\sp{n}  \}
\quad (x \in \ZZ\sp{n}) ,
\end{equation}
where it is assumed that the infimum on the right-hand side
is bounded from below (i.e., $\not= -\infty$) for every $x \in \ZZ\sp{n}$.
We have
\begin{equation}  \label{f1f2convdom}
 \dom(f_{1} \convZ f_{2}) = \dom f_{1} + \dom f_{2}  .
\end{equation}
Let
$\delta_{S_{1}}, \delta_{S_{2}} : \ZZ\sp{n} \to \{0, + \infty\}$ 
be the indicator functions of $S_{1}, S_{2} \subseteq \ZZ\sp{n}$, respectively.
Then their convolution $\delta_{S_{1}} \convZ \delta_{S_{2}}$ 
coincides with the indicator function $\delta_{S_{1}+S_{2}}$ of 
the Minkowski sum $S_{1}+S_{2}$, i.e., 
\begin{equation}  \label{setconvolMinkow}
 \delta_{S_{1}} \convZ \delta_{S_{1}} = \delta_{S_{1}+S_{2}}.
\end{equation}

\subsection{Results for integrally convex functions}
\label{SCicconvolresult}

It is known \cite{Mdcasiam,MS01rel}
that the convolution of two integrally convex functions
is not necessarily integrally convex.
This is demonstrated
 by the following example \cite[Example 3.15]{Mdcasiam}
showing that the Minkowski sum of integrally convex sets
is not necessarily integrally convex.

\begin{example} \rm \label{EXicdim2sumhole}
The Minkowski sum of
$S_{1} = \{ (0,0), (1,1) \}$
and
$S_{2} = \{ (1,0), (0,1) \}$
is equal to 
$S_{1}+S_{2} = \{ (1,0), (0,1), (2,1), (1,2) \}$,
which has a ``hole'' at $(1,1)$, i.e.,
$(1,1) \in \overline{S_{1}+S_{2}}$ and
$(1,1) \not\in S_{1}+S_{2}$.
Both $S_{1}$ and $S_{2}$
are integrally convex,
but $S_{1}+S_{2}$ is not integrally convex.
\finbox
\end{example}

Thus we are motivated to consider
the convolution of an integrally convex function
and a separable convex function.
We denote a separable convex function by $\varphi$, i.e.,
\begin{equation} \label{sepconvdef1}
\varphi(x) = \sum_{i=1}\sp{n} \varphi_{i}(x_{i}) ,
\end{equation}
where
$\varphi_{i}: \ZZ \to \RR \cup \{ +\infty \}$
is a univariate discrete convex function for $i=1,2,\ldots,n$.
Also we always use $B$ to denote an integer interval (or box),
i.e.,
$B = [a,b]_{\ZZ}$
for some integer vectors 
$a \in (\ZZ \cup \{ -\infty \})\sp{n}$ and 
$b \in (\ZZ \cup \{ +\infty \})\sp{n}$ 
with $a \leq b$.

The following theorems are the main results of this section,
dealing with sets and functions, respectively.

\begin{theorem} \label{THicsetbox}
The Minkowski sum $S+B$ of an integrally convex set $S$ and an integer interval $B$
is an integrally convex set.
\end{theorem}
\begin{proof}
The proof is given in Section~\ref{SCprficsetbox}.
\end{proof}

\begin{theorem} \label{THicsepconvol}
The convolution $f \convZ \varphi$ of an integrally convex function $f$
and a separable convex function $\varphi$ is an integrally convex function.
\end{theorem}
\begin{proof}
While the details are given in Section~\ref{SCprficsepconvol},
we mention here that the proof consists of two steps.
\begin{enumerate}
\item
If the effective domains of $f$ and $\varphi$ are bounded, 
we can use Theorem~\ref{THintconvfnsetchar}
to reduce the proof of Theorem~\ref{THicsepconvol} to 
Theorem~\ref{THicsetbox} for integrally convex sets.

\item
In the general case with possibly unbounded 
effective domains, 
we consider sequences $\{ f_{k} \}$ and $\{ \varphi_{k} \}$
with bounded effective domains,
which are constructed from $f$ and $\varphi$ 
as their restrictions to finite intervals.
Step~1 shows that $f_{k} \convZ \varphi_{k}$ is integrally convex for each $k$.
Then 
the integral convexity of $f \convZ \varphi$
is established by a limiting argument.
\end{enumerate}
\end{proof}

\begin{remark} \rm \label{RMuseconvoldist}
It follows from Theorem~\ref{THicsepconvol} that 
the $\ell_{1}$-distance $d\sp{(1)}$ and the squared $\ell_{2}$-distance $d\sp{(2)}$
to an integrally convex set $S$
are both integrally convex, where
$d\sp{(1)}$ and $d\sp{(2)}$ are defined respectively as
\begin{align}
d\sp{(1)}(x) &= 
\inf \{   \| x-y \|_{1}  \mid y \in S \}
\qquad (x \in \mathbb{Z}\sp{n}),
\label{distfn1Z}
\\
d\sp{(2)}(x) &= 
\inf \{ {\| x-y \|_{2}}\sp{2}  \mid y \in S \}
\qquad (x \in \mathbb{Z}\sp{n}) .
\label{distfn2Z}
\end{align}
Indeed, the indicator function $\delta_{S}$ of $S$
is integrally convex,  
both
$\varphi\sp{(1)}(x) = \| x \|_{1}$  
and 
$\varphi\sp{(2)}(x) = {\| x \|_{2}}\sp{2}$  
are separable convex, and
$d\sp{(k)} = \delta_{S} \convZ \varphi\sp{(k)}$ for $k=1,2$
by  (\ref{distfn1Z}) and (\ref{distfn2Z}).
Furthermore, with a parameter $a > 0$, we can define
penalty functions 
$g_{a}\sp{(k)}(x) = a \ d\sp{(k)}(x)$  $(k=1,2)$
associated with $S$.
Indeed, $g_{a} = g_{a}\sp{(k)}$ with $k \in \{ 1,2 \}$ satisfies
(i) 
$\dom g_{a} = \mathbb{Z}\sp{n}$,
(ii)
$g_{a}(x) \geq 0$ for all $x \in \mathbb{Z}\sp{n}$,
(iii)
$g_{a}(x)=0 \Leftrightarrow x \in S$,
for each $a>0$,
and (iv)
$\displaystyle
\lim_{a \to +\infty} g_{a}(x) = +\infty$
for all $x \not\in S$.
\finbox
\end{remark}

\begin{remark} \rm \label{RMuseconvolextend}
It also follows from Theorem~\ref{THicsepconvol} 
that any integrally convex function 
$f: \ZZ\sp{n} \to \RR \cup\{ +\infty  \}$,
defined effectively on a subset $S$ of $\ZZ\sp{n}$,
can be extended to another integrally convex function 
that takes finite values on the entire integer lattice $\ZZ\sp{n}$.
To be specific, with a parameter $a > 0$, we define
\begin{align}
g_{a}\sp{(1)}(x) &= \inf \{ f(y) + a \ \| x-y \|_{1}  \mid y \in \dom f \}
\qquad (x \in \mathbb{Z}\sp{n}),
\label{domext1Z}
\\
g_{a}\sp{(2)}(x) &= \inf \{ f(y) + a \  {\| x-y \|_{2}}\sp{2}  \mid y \in \dom f \}
\qquad (x \in \mathbb{Z}\sp{n}) .
\label{domext2Z}
\end{align}
Since 
$\varphi\sp{(1)}(x) = \| x \|_{1}$  
and 
$\varphi\sp{(2)}(x) = {\| x \|_{2}}\sp{2}$  
are separable convex,
both 
$g_{a}\sp{(1)}$ and
$g_{a}\sp{(2)} $ are integrally convex by Theorem~\ref{THicsepconvol}.
Moreover, 
$g_{a} = g_{a}\sp{(k)}$ with $k \in \{ 1,2 \}$ satisfies
(i) 
$\dom g_{a} = \ZZ\sp{n}$,
(ii) 
for each $x \in \dom f$, 
there exists $\alpha(x) >0$ such that
$g_{a}(x) = f(x)$
for all $a \geq \alpha(x)$,
and
(iii)
if $\dom f$ is bounded, 
there exists $\hat \alpha >0$ such that
$g_{a}(x)  = f(x)$ for all $x \in \dom f$ and $a \geq \hat \alpha$.
\finbox
\end{remark}

\begin{remark} \rm \label{RMconvolMinkow}
Consider the indicator functions $\delta_S, \delta_B: \ZZ^n \to \{0, + \infty\}$ 
of $S, B \subseteq  \ZZ^n$.
Since their convolution $\delta_S \convZ \delta_B$ 
coincides with the indicator function $\delta_{S+B}$ of 
the Minkowski sum $S+B$ by (\ref{setconvolMinkow}),
Theorem~\ref{THicsetbox} is a special case of Theorem~\ref{THicsepconvol}.
It is emphasized, however, that our proof of
Theorem~\ref{THicsepconvol} uses Theorem~\ref{THicsetbox}.
\finbox
\end{remark}

\begin{remark} \rm \label{RMprojFromConvol}
The projection operation can be regarded as a special case of
the convolution operation.
Let
$g(x) = \inf \{f(x,y)\mid y\in \mathbb{Z}\sp{m}\}$
be the projection of 
$f: \mathbb{Z}\sp{n+m} \to \mathbb{R} \cup \{ +\infty  \}$.
Consider $B = \{ (x,y) \in \ZZ\sp{n+m} \mid x = \veczero \}$,
which is an integer interval 
$[ (\veczero,-\infty),(\veczero,+\infty) ]_{\ZZ}$ in $\ZZ\sp{n+m}$.
Then the projection $g$ coincides with the convolution $f \convZ \delta_{B}$
in the sense that
$g(x) = (f \convZ \delta_{B})(x, \veczero)$
for $x \in \ZZ\sp{n}$,  
which we denote as
$g = (f \convZ  \delta_{B}) |_{\ZZ\sp{n}}$.
If $f$ is integrally convex,
$f \convZ  \delta_{B}: \mathbb{Z}\sp{n+m} \to \mathbb{R} \cup \{ +\infty  \}$
is also integrally convex by Theorem~\ref{THicsepconvol}, 
and hence 
$(f \convZ  \delta_{B}) |_{\ZZ\sp{n}}=g$
is also integrally convex.
This argument serves as an alternative proof of Theorem~\ref{THicvproj}.
\finbox
\end{remark}

\subsection{Results for discrete midpoint convex functions}
\label{SCmidptconvolresult}

The convolution operation is not amenable to 
discrete midpoint convexity.
In this section we demonstrate this in terms of examples.

We first note the following example
(\cite[Example 4.11]{MS01rel}, \cite[Note 5.11]{Mdcasiam})
about the Minkowski sum of L$^{\natural}$-convex sets.

\begin{example} \rm \label{EXlnatsetsum}
The Minkowski sum of
$S_{1} =  \{(0, 0, 0), (1, 1, 0)\}$ and $S_{2} =  \{(0, 0, 0),$ $(0, 1, 1)\}$
is equal to 
$S_{1} + S_{2} = \{(0, 0, 0), (0, 1, 1), (1, 1, 0), (1, 2, 1)\}$.
For $x=(0, 1, 1)$ and $y=(1, 1, 0)$ in $S_{1} + S_{2}$,
we have
$\left\lceil (x+y)/2 \right\rceil = (1, 1, 1) \not\in S_{1} + S_{2}$
and
$\left\lfloor (x+y)/2 \right\rfloor = (0, 1, 0) \not\in S_{1} + S_{2}$.
Both $S_{1}$ and $S_{2}$
are L$^{\natural}$-convex,
but $S_{1}+S_{2}$ is not L$^{\natural}$-convex.
\finbox
\end{example}

For the Minkowski sum we observe the following.

\begin{itemize}
\item
The Minkowski sum of two L$^{\natural}$-convex sets
is not necessarily L$^{\natural}$-convex,
though it is integrally convex.
This is shown by Example~\ref{EXlnatsetsum} above
and \cite[Theorem 8.42]{Mdcasiam}.

\item
The Minkowski sum of two discrete midpoint convex sets
is not necessarily discrete midpoint convex (nor integrally convex).
This is shown by Example~\ref{EXicdim2sumhole}.

\item
The Minkowski sum $S+B$ of a discrete midpoint convex set $S$
and an integer interval $B$ is not necessarily discrete midpoint convex,
though it is integrally convex.
This is shown by Example~\ref{EXdicdim3set} below
and Theorem~\ref{THicsetbox}.
\end{itemize}

\begin{example}  \rm \label{EXdicdim3set}
The Minkowski sum of
$S = \{ (0,0,1), (1,1,0) \}$
and
$B = \{ (0,0,0), (1,0,0) \}$
is equal to 
$S+B = \{ (0,0,1), (1,1,0),  (1,0,1), (2,1,0) \}$.
$S$ is discrete midpoint convex
and $B$ is an integer interval.
For
$x = (0,0,1) \in S+B$ and $y = (2,1,0) \in S+B$ we have
\[
\| x - y \|_{\infty} = 2,
\quad
\left\lceil \frac{x+y}{2} \right\rceil 
= (1,1,1) \not\in S+B, 
\quad
\left\lfloor \frac{x+y}{2} \right\rfloor
= (1,0,0) \not\in S+B.
\]
Therefore, $S+B$ is not discrete midpoint convex.
\finbox
\end{example}

Next we turn to discrete midpoint convex functions.
Recall the relation (\ref{setconvolMinkow})
between the Minkowski sum of sets and 
the convolution of their indicator functions.
Our observations above about the Minkowski sums imply the following.

\begin{itemize}
\item
The convolution $f_{1} \convZ f_{2}$ of globally (resp., locally) discrete midpoint convex functions 
$f_{1}, f_{2}$ is not necessarily globally (resp., locally) discrete midpoint convex
 (nor integrally convex).

\item
The convolution $f \convZ \varphi$ of a globally discrete midpoint convex function $f$ 
and a separable convex function $\varphi$ is 
not necessarily globally discrete midpoint convex,
though it is integrally convex by (\ref{fnclasses3}) and Theorem~\ref{THicsepconvol}.

\item
The convolution $f \convZ \varphi$ of a locally discrete midpoint convex function $f$ 
and a separable convex function $\varphi$ is 
not necessarily locally discrete midpoint convex,
though it is integrally convex by (\ref{fnclasses3}) and Theorem~\ref{THicsepconvol}.
\end{itemize}

We show another example of $f$ for the above statements
such that the effective domain of $f$ is an integer interval.

\begin{example}  \rm \label{EXdicdim3fn}
Let $S = \{ (0,0,1), (1,1,0) \}$,
$B = \{ (0,0,0), (1,0,0) \}$, and
$\varphi= \delta_B$, and define
$f: \ZZ^3 \to \mathbb{R} \cup \{ +\infty\}$
with  
$\dom f = [ (0,0,0), (1,1,1) ]_{\ZZ}$
by
\[
f(x)  =
   \left\{  \begin{array}{ll}
    0            &   (x \in S) ,     \\
    1       &   (x \in [ (0,0,0), (1,1,1) ]_{\ZZ} \setminus S) . \\
                      \end{array}  \right.
\]
The convolution $f \convZ \varphi$,
with
$\dom (f \convZ \varphi) = [ (0,0,0), (2,1,1) ]_{\ZZ}$, is given by 
\[
(f \convZ \varphi)(x)  =
   \left\{  \begin{array}{ll}
    0            &   (x \in S+B) ,     \\
    1       &   (x \in [  (0,0,0), (2,1,1) ]_{\ZZ} \setminus (S+B)) . \\
                      \end{array}  \right.
\]
For $x = (0,0,1)$, $y = (2,1,0)$ we have 
$\| x - y \|_{\infty} = 2$,
\[
(f \convZ \varphi)(x) = (f \convZ \varphi)(y) =0,
\quad
(f \convZ \varphi) \left(
\left\lceil \frac{x+y}{2} \right\rceil 
\right)
= 
(f \convZ \varphi) \left(
\left\lfloor \frac{x+y}{2} \right\rfloor
\right) = 1 .
\]
Hence 
 $f \convZ \varphi$ is not (globally or locally) discrete midpoint convex.
\finbox
\end{example}

By featuring discrete midpoint convexity (\ref{lnatfmidconv})
we can recast our knowledge as follows.

\begin{itemize}
\item
Discrete midpoint convexity 
for all $(x, y)$ with
$\| x - y \|_{\infty} \geq 2$
is not preserved in the transformation $f \mapsto f \convZ \varphi$.

\item
Discrete midpoint convexity 
for all $(x, y)$ with
 $\| x - y \|_{\infty} = 2$
is not preserved in the transformation $f \mapsto f \convZ \varphi$.

\item
Discrete midpoint convexity 
for all $(x, y)$ with
$\| x - y \|_{\infty} \geq 1$
is preserved in the transformation $f \mapsto f \convZ \varphi$.
\end{itemize}
The last statement is a reformulation of
the following (known) fact \cite[Theorem 7.11]{Mdcasiam}.

\begin{theorem} \label{THlsepconvol}
The convolution $f \convZ \varphi$ of an L$^{\natural}$-convex function $f$
and a separable convex function $\varphi$ is an L$^{\natural}$-convex function.
\finbox
\end{theorem}
In Section~\ref{SClnatsepconvol} we give a direct proof of this theorem;
the proof in \cite{Mdcasiam} uses a reduction to an L-convex function,
which is defined 
in terms of submodularity and linearity in the direction of $(1,1,\ldots,1)$.

\subsection{Proof of Theorem~\ref{THicsetbox}}
\label{SCprficsetbox}

We prove Theorem~\ref{THicsetbox}
that the Minkowski sum $S+B$ of an integrally convex set $S$ 
and an integer interval $B$ is integrally convex.

Let $S \subseteq  \mathbb{Z}\sp{n}$ be an integrally convex set
and $B = [a,b]_{\ZZ} \subseteq  \mathbb{Z}\sp{n}$
with
$a \in (\ZZ \cup \{ -\infty \})\sp{n}$
and $b \in (\ZZ \cup \{ +\infty \})\sp{n}$.
For $i=1,2,\ldots,n$
we denote the $i$th unit vector by $\unitvec{i} \in \mathbb{Z}\sp{n}$ and 
put $B_{i}= \{ t \unitvec{i} \mid a_{i} \leq t \leq b_{i}, t \in \ZZ \}$.
For $i=1$, for example, 
$B_{1}= \{ (t,0,\ldots,0)  \mid a_{1} \leq t \leq b_{1}, t \in \ZZ \}$.
Then $B = B_{1} + B_{2} + \cdots + B_{n}$,
and hence
$S+B = (\cdots ((S+ B_{1}) + B_{2}) + \cdots ) + B_{n}$.
Thus the proof of Theorem~\ref{THicsetbox} is reduced to showing the following lemma.

\begin{lemma} \label{LMicsetline}
The Minkowski sum $S+B$ of an integrally convex set $S$ and 
$B= \{ (t,0,\ldots,0)  \mid \hat a \leq t \leq \hat b, t \in \ZZ \}$
is integrally convex, 
where $\hat a \in \ZZ \cup \{ -\infty \}$ and $\hat b \in \ZZ \cup \{ +\infty \}$.
\end{lemma}
\begin{proof}
Let $x \in \overline{S+B}$.
Our goal is to show $x \in \overline{(S+B) \cap N(x)}$.
The proof goes as follows.

Since 
$\overline{S+B}= \overline{S} + \overline{B}$
(see, e.g.,  \cite[Proposition 3.17]{Mdcasiam})
we can represent $x$ as $x = y+z$
with $y \in \overline{S}$ and $z \in \overline{B}$.
Since $S$ is integrally convex, 
the vector $y$ can be represented as
$ y = \sum_{k=1}\sp{l} \lambda_{k} y\sp{(k)}$
with some $y\sp{(k)} \in S \cap N(y)$ 
and $\lambda_{k} > 0$
$(k=1,2,\ldots,l)$,
where
$\sum_{k=1}\sp{l} \lambda_{k}  = 1$.
The other vector $z$ can be represented as
$z = (\zeta + \beta,0,\ldots,0)$
with some $\zeta \in \ZZ$ and $\beta \in \RR$, where $0 \leq \beta < 1$.
We show $x \in \overline{(S+B) \cap N(x)}$ by finding vectors
$v_i\in(S+B)\cap N(x)$ and coefficients $\mu_i$ of convex combination
$(\sum \mu_i =1,\mu_i\geq 0)$
such that 
\[
x  =  
   y + z  
  =  
  \sum_{k=1}\sp{l} \lambda_{k} y\sp{(k)} + (\zeta + \beta)e
  =  \sum \mu_i v_i.
\]
By $x = y+z$ we have
$x_{1} - y_{1} =z_{1} = \zeta + \beta$,
which implies
$0 \leq x_{1} - y_{1} -  \zeta < 1$. 
We divide into two cases:
Case 1: 
$\lfloor x_{1} \rfloor -  \zeta  \leq y_{1}$,
and
Case 2: 
$ \lfloor x_{1} \rfloor -  \zeta > y_{1}$.

\medskip

{\bf Case 1 ($\lfloor x_{1} \rfloor -  \zeta  \leq y_{1}$)}. 
We have
$\lfloor x_{1} \rfloor -  \zeta  \leq y_{1} \leq x_{1}-  \zeta  \leq  \lceil x_{1} \rceil -  \zeta$,
and hence
$N(x) - \zeta \unitvecfirst \supseteq N(y)$,
where $\unitvecfirst = \unitvec{1} = (1,0,\ldots,0)$.
Since
$y\sp{(k)} \in  N(y) \subseteq N(x) - \zeta \unitvecfirst$,
we have
$y\sp{(k)}_{1} = \lfloor x_{1} \rfloor -  \zeta$ or $\lceil x_{1} \rceil -  \zeta$
$(k=1,2,\ldots,l)$.
After renumbering, if necessary, we may assume
\begin{align}
y\sp{(k)}_{1} &= \lfloor x_{1} \rfloor -  \zeta
\quad 
(k=1,2,\ldots,k_{0}),
\label{Case1ykd}
\\
y\sp{(k)}_{1} &= \lceil x_{1} \rceil -  \zeta
\quad 
(k=k_{0}+1,k_{0}+2,\ldots,l ),
\label{Case1yku}
\end{align}
where $k_{0}=0$ if $x_{1} \in \ZZ$.
Then
\begin{align*}
 y_{1} &=  
 \sum_{k=1}\sp{k_{0}} \lambda_{k} y\sp{(k)}_{1}
  +  \sum_{k=k_{0}+1}\sp{l} \lambda_{k} y\sp{(k)}_{1}
= 
 \sum_{k=1}\sp{k_{0}} \lambda_{k} \lfloor x_{1} \rfloor 
 +  \sum_{k=k_{0}+1}\sp{l} \lambda_{k} \lceil x_{1} \rceil -  \zeta .
\end{align*}
With the use of this expression we obtain
\begin{align}
 \beta &= x_{1} - \zeta - y_{1}
\notag \\ &= 
x_{1} - \sum_{k=1}\sp{k_{0}} \lambda_{k} \lfloor x_{1} \rfloor 
 -  \sum_{k=k_{0}+1}\sp{l} \lambda_{k} \lceil x_{1} \rceil 
\notag \\ &= 
x_{1} - \sum_{k=1}\sp{k_{0}} \lambda_{k} (\lceil x_{1} \rceil -1 ) 
 -  \sum_{k=k_{0}+1}\sp{l} \lambda_{k} \lceil x_{1} \rceil 
\quad \mbox{(since $k_{0}=0$ if $x_{1} \in \ZZ$)}
\notag \\ &= 
x_{1} - \lceil x_{1} \rceil + \sum_{k=1}\sp{k_{0}} \lambda_{k}
\notag \\ &\leq
 \sum_{k=1}\sp{k_{0}} \lambda_{k}.
\label{Case1beta}
\end{align}
Let $k_{1}$ be the minimum $k'$ satisfying
$\beta \leq  \sum_{k=1}\sp{k'} \lambda_{k}$.
We have $k_{1} \leq k_{0}$ by (\ref{Case1beta}).
It should be clear that $k_{1}=0$ if $\beta =0$.
Define
$\alpha = \beta -  \sum_{k=1}\sp{k_{1}-1} \lambda_{k}$,
where  $\alpha = 0$ if $\beta =0$.
We have $0 \leq \alpha \leq \lambda_{k_{1}}$.
Using $\beta = \alpha + \sum_{k=1}\sp{k_{1}-1} \lambda_{k}$
and 
$\sum_{k=1}\sp{l} \lambda_{k}  = 1$
we obtain
\begin{align}
 x &= y + z
\notag \\ &= 
 \sum_{k=1}\sp{l} \lambda_{k} y\sp{(k)}
 + (\zeta + \beta) \unitvecfirst
\notag \\ &= 
 \sum_{k=1}\sp{k_{1}-1}  \lambda_{k} y\sp{(k)} 
+ \lambda_{k_1}y\sp{(k_{1})} + \sum_{k=k_{1}+1}\sp{l} \lambda_{k} y\sp{(k)} 
+ (\zeta +\alpha + \sum_{k=1}\sp{k_{1}-1} \lambda_{k}) \unitvecfirst
\notag \\ &= 
 \sum_{k=1}\sp{k_{1}-1}  \lambda_{k} (y\sp{(k)} + (\zeta +1) \unitvecfirst) 
 + \alpha  (y\sp{(k_{1})} + (\zeta +1) \unitvecfirst) 
\notag \\ & \phantom{=} \ {}
 + (\lambda_{k_{1}}-\alpha ) (y\sp{(k_{1})} + \zeta  \unitvecfirst) 
 + \sum_{k=k_{1}+1}\sp{l} \lambda_{k} (y\sp{(k)} + \zeta  \unitvecfirst).
\label{Case1xyz}
\end{align}
Since $0 \leq \alpha \leq \lambda_{k_{1}}$
and 
$\displaystyle
 \sum_{k=1}\sp{k_{1}-1}  \lambda_{k}  + \alpha + (\lambda_{k_{1}}-\alpha ) 
 + \sum_{k=k_{1}+1}\sp{l} \lambda_{k} 
 =1$,
(\ref{Case1xyz}) represents $x$ as a convex combination of 
\begin{align}
& y\sp{(k)} + (\zeta +1) \unitvecfirst \ \ (k=1,\ldots, k_{1}-1),
\quad
y\sp{(k_{1})} + (\zeta +1) \unitvecfirst,
\label{Case1MinkowVector1}
\\ &
y\sp{(k_{1})} + \zeta  \unitvecfirst,
\label{Case1MinkowVector2}
 \\ &
y\sp{(k)} + \zeta  \unitvecfirst \ \ (k=k_{1}+1,\ldots, l) ,
\label{Case1MinkowVector3}
\end{align}
where the vectors in (\ref{Case1MinkowVector1}) and (\ref{Case1MinkowVector2}) 
are missing when $\beta=0$ (which implies $k_{1}=0$).
Then $x \in \overline{(S+B) \cap N(x)}$ follows from Claim 1 below.

Claim 1:
All vectors in (\ref{Case1MinkowVector1})--(\ref{Case1MinkowVector3}) 
belong to $(S+B) \cap N(x)$.
\\
(Proof)
Since $y\sp{(k)} \in S$ and
$\zeta \unitvecfirst \in B$,
the vectors in (\ref{Case1MinkowVector2}) and (\ref{Case1MinkowVector3}) 
belong to $S+B$.
The vectors in (\ref{Case1MinkowVector1}) exist only when $\beta  > 0$,
in which case $(\zeta +1) \unitvecfirst \in B$
and hence
the vectors in (\ref{Case1MinkowVector1}) belong to $S+B$.
By $N(x) \supseteq  N(y) + \zeta \unitvecfirst$,
the vectors in
 (\ref{Case1MinkowVector2}) 
and  (\ref{Case1MinkowVector3}) obviously belong to $N(x)$.
The vectors in (\ref{Case1MinkowVector1}) also belong to $N(x)$
since $x=y+(\zeta + \beta, 0,\ldots,0)$ and 
\[
y_1\sp{(k)} +\zeta +1 = \lfloor x_1 \rfloor +1,\quad
y_i\sp{(k)} =x_i \quad(i\neq 1)
\]
for $k=1,\ldots, k_0$
by (\ref{Case1ykd}) and $k_{1} \leq k_{0}$.
Thus Claim 1 is proved.

\medskip

{\bf Case 2 ($y_{1} < \lfloor x_{1} \rfloor -  \zeta$)}.
We have
$\lfloor x_{1} \rfloor -  \zeta -1 \leq x_{1} -\zeta - 1 < y_{1} < \lfloor x_{1} \rfloor -  \zeta$
and hence
$N(x) - (\zeta +1) \unitvecfirst \supseteq   N(y)$.
Since
$y\sp{(k)} \in  N(y) \subseteq   N(x) - (\zeta +1) \unitvecfirst$,
we have 
$y\sp{(k)}_{1} = \lfloor x_{1} \rfloor -  (\zeta +1)$ 
or $\lceil x_{1} \rceil - (\zeta +1)$
$(k=1,2,\ldots,l)$.
This implies $x_{1} \not\in \ZZ$,
since otherwise (i.e, if $x_{1} \in \ZZ$), 
$y\sp{(k)}_{1} = x_{1}- (\zeta +1)$
for all $k$, and 
$y_{1} =  \sum_{k=1}\sp{l} \lambda_{k} y\sp{(k)}_{1}  
= x_{1}- (\zeta +1)$,
contradicting $\beta = x_{1} - y_{1} - \zeta  <1 $.
After renumbering, if necessary, we may assume
\begin{align}
y\sp{(k)}_{1} &= \lfloor x_{1} \rfloor -  (\zeta +1)
\quad
(k=1,2,\ldots,k_{0}),
\label{Case2ykd}
\\
y\sp{(k)}_{1} &= \lceil x_{1} \rceil -  (\zeta +1)
\quad 
(k=k_{0}+1,k_{0}+2,\ldots,l),
\label{Case2yku}
\end{align}
where $1 \leq k_{0} \leq l-1$.
We have $\beta > 0$, since
 $\beta = x_{1} - y_{1} - \zeta  >  x_{1} - \lfloor x_{1} \rfloor \geq 0$.

It follows from (\ref{Case2ykd}) and (\ref{Case2yku}) that
\begin{align*}
 y_{1} &=  
\sum_{k=1}\sp{k_{0}} \lambda_{k} y\sp{(k)}_{1}
 +  \sum_{k=k_{0}+1}\sp{l} \lambda_{k} y\sp{(k)}_{1}
= 
 \sum_{k=1}\sp{k_{0}} \lambda_{k} \lfloor x_{1} \rfloor 
 +  \sum_{k=k_{0}+1}\sp{l} \lambda_{k} \lceil x_{1} \rceil -  (\zeta +1) .
\end{align*}
With the use of this expression we obtain
\begin{align}
 \beta &= x_{1} - \zeta  - y_{1}
\notag \\ &= 
x_{1} +1 
 - \sum_{k=1}\sp{k_{0}} \lambda_{k} \lfloor x_{1} \rfloor 
 -  \sum_{k=k_{0}+1}\sp{l} \lambda_{k} \lceil x_{1} \rceil   
\notag \\ &= 
x_{1} + 1 - \sum_{k=1}\sp{k_{0}} \lambda_{k} (\lceil x_{1} \rceil -1 ) 
 -  \sum_{k=k_{0}+1}\sp{l} \lambda_{k} \lceil x_{1} \rceil 
\qquad \mbox{(since $x_{1} \not\in \ZZ$)}
\notag \\ &= 
x_{1} +1 - \lceil x_{1} \rceil + \sum_{k=1}\sp{k_{0}} \lambda_{k}
\notag \\ &\geq
 \sum_{k=1}\sp{k_{0}} \lambda_{k}.
\label{Case2beta}
\end{align}
Let $k_{1}$ be the maximum $k'$ satisfying
$\sum_{k=1}\sp{k'} \lambda_{k} \leq \beta$.
We have $k_{1} \leq l -1$,
since 
$\sum_{k=1}\sp{l} \lambda_{k} =1$
and $\beta < 1$,
and $k_{1} \geq k_{0}$ by (\ref{Case2beta}).
Therefore,
$1 \leq k_{1} \leq l -1$.
Define
$\alpha = \beta - \sum_{k=1}\sp{k_{1}} \lambda_{k} $,
which satisfies
$0 \leq \alpha \leq \lambda_{k_{1}+1}$. 
Using 
$\beta = \alpha +  \sum_{k=1}\sp{k_{1}} \lambda_{k} $
we obtain
\begin{align}
 x &= y + z
\notag \\ &= 
 \sum_{k=1}\sp{l} \lambda_{k} y\sp{(k)}
 + (\zeta + \beta) \unitvecfirst
\notag \\ &= 
 \sum_{k=1}\sp{k_{1}}  \lambda_{k} (y\sp{(k)} + (\zeta +1) \unitvecfirst) 
 + \alpha  (y\sp{(k_{1}+1)} + (\zeta +1) \unitvecfirst) 
\notag \\ & \phantom{=} \ {}
 + (\lambda_{k_{1}+1}-\alpha ) (y\sp{(k_{1}+1)} + \zeta  \unitvecfirst) 
 + \sum_{k=k_{1}+2}\sp{l} \lambda_{k} (y\sp{(k)} + \zeta  \unitvecfirst) .
\label{Case2xyz}
\end{align}
Since 
$0 \leq \alpha \leq \lambda_{k_{1}+1}$
and 
$\displaystyle 
 \sum_{k=1}\sp{k_{1}}  \lambda_{k}  + \alpha + (\lambda_{k_{1}+1}-\alpha ) 
 + \sum_{k=k_{1}+2}\sp{l} \lambda_{k} =1$,
(\ref{Case2xyz}) represents $x$ as a convex combination of 
\begin{align}
&
y\sp{(k)} + (\zeta +1) \unitvecfirst \ \ (k=1,\ldots, k_{1}),
\quad
y\sp{(k_{1}+1)} + (\zeta +1) \unitvecfirst,
\label{Case2MinkowVector1}
 \\ &
y\sp{(k_{1}+1)} + \zeta  \unitvecfirst,
\label{Case2MinkowVector2}
 \\ &
y\sp{(k)} + \zeta  \unitvecfirst \ \ (k=k_{1}+2,\ldots, l),
\label{Case2MinkowVector3}
\end{align}
where the vectors in (\ref{Case2MinkowVector3}) 
are missing when $k_{1}=l-1$.
Then $x \in \overline{(S+B) \cap N(x)}$ follows from Claim 2 below.

Claim 2:
All vectors in (\ref{Case2MinkowVector1})--(\ref{Case2MinkowVector3}) 
belong to $(S+B) \cap N(x)$.
\\
(Proof)
Since $y\sp{(k)} \in S$ and 
$\{ \zeta \unitvecfirst, (\zeta +1) \unitvecfirst \} \subseteq B$
as a consequence of $\beta > 0$,
all the vectors in (\ref{Case2MinkowVector1})--(\ref{Case2MinkowVector3}) 
belong to $S+B$.
By $N(x) \supseteq   N(y) + (\zeta +1)  \unitvecfirst $,
the vectors in  (\ref{Case2MinkowVector1}) obviously belong to $N(x)$.
The vectors in 
(\ref{Case2MinkowVector2}) and
(\ref{Case2MinkowVector3}) 
 also belong to $N(x)$
since $x=y+(\zeta + \beta, 0,\ldots,0)$ and 
\[
y_1\sp{(k)} +\zeta = \lceil x_1 \rceil -1,\quad
y_i\sp{(k)} =x_i \quad(i\neq 1)
\]
for $k=k_0+1,k_0+2,\ldots, l$
by (\ref{Case2yku}) and $k_{1} \geq k_{0}$.
Thus Claim 2 is proved.

In both cases, Case 1 and Case 2, 
we have arrived at
$x \in \overline{(S+B) \cap N(x)}$.
This completes the proof of Lemma~\ref{LMicsetline}.
\end{proof}

\subsection{Proof of Theorem~\ref{THicsepconvol}}
\label{SCprficsepconvol}

In this section we prove 
Theorem~\ref{THicsepconvol}
that the convolution $f \convZ \varphi$ 
of an integrally convex function $f$
and a separable convex function $\varphi$
is integrally convex.

The proof consists of two steps.
In Step~1, we prove integral convexity of 
$g = f \convZ \varphi$ 
when $\dom f$ and $\dom \varphi$ are bounded.
In Step~2, we cope with the general case
by considering sequences $\{ f_{k} \}$ and $\{ \varphi_{k} \}$,
with $\dom f_{k}$ and $\dom \varphi_{k}$ bounded, that converge
to $f$ and $\varphi$,  respectively.
We first note that $\dom g$ is an integrally convex set
by Theorem~\ref{THicsetbox},
since
$\dom g = \dom f + \dom \varphi$,
in which
$\dom f$ is an integrally convex set
and $\dom \varphi$ is an integer interval.

Step~1:
We assume that 
$\dom f$ and $\dom \varphi$ are bounded.
Then $\dom g = \dom f + \dom \varphi$ is also bounded.
Let $p \in \RR\sp{n}$ and note
\begin{align*}
g[-p]  &=  f[-p] \convZ \varphi[-p],  \\
\argmin g[-p] &= \argmin f[-p] + \argmin \varphi[-p] .
\end{align*}
Since
 $\argmin f[-p]$ is an integrally convex set
by Theorem~\ref{THintconvfnsetchar} (``only if'')
and 
$\argmin \varphi[-p]$ is an integer interval,
$\argmin g[-p]$ is an integrally convex set by 
Theorem~\ref{THicsetbox}.
Then Theorem~\ref{THintconvfnsetchar} (``if'')
shows that $g$ is an integrally convex function.

Step~2:
For $k=1,2,\ldots$, let $f_{k}$ denote the function
obtained from $f$ by restricting the effective domain to
the integer interval 
$\{ x  \in \ZZ\sp{n} \mid \| x \|_{\infty} \leq k \}$;
define $\varphi_{k}$ from $\varphi$ in a similar manner.
For each $k$, 
$f_{k}$ is an integrally convex function by Theorem~\ref{THfavtarProp33}
and $\varphi_{k}$ is a separable convex function,
both with bounded effective domains%
\footnote{
We only need to consider sufficiently large $k$
for which $\dom f_{k}$ and $\dom \varphi_{k}$ are nonempty. 
}, 
and therefore
$g_{k}  =  f_{k} \convZ \varphi_{k}$ is 
integrally convex by Step~1.
By Lemma \ref{LMconvolimit} below,
$\{ g_{k}(x) \}_{k}$ converges to $g(x)$
for each $x \in \dom g$.
Therefore $g$ is integrally convex since the limit of 
integrally convex functions is integrally convex by Lemma \ref{LMicflimit} below.

\begin{lemma} \label{LMconvolimit}
For each $x \in \dom g$,
the sequence $\{ g_{k}(x) \}_{k}$
is nonincreasing and converges to $g(x)$. 
\end{lemma}
\begin{proof}
Since
\begin{align*} \label{gphiconvoldef2}
g_{k}(x)  &=  
  \inf\{ f_{k}(y) + \varphi_{k}(z) \mid    x= y + z  \}
\\
  &=  
  \inf\{ f(y) + \varphi(z) \mid
      x= y + z, \,  \| y \|_{\infty} \leq k, \, \| z \|_{\infty} \leq k
 \} ,
\end{align*}
the sequence $\{ g_{k}(x) \}_{k}$
is obviously nonincreasing, while it is bounded from below by our 
standing assumption stated at the beginning of Section~\ref{SCconvol}.
Therefore, the limit $\lim_{k \to \infty} g_{k}(x)$ exists.
The limit is obviously equal to $g(x)$, but we give a formal proof for completeness.
Let $\varepsilon > 0$ be any positive number.
By the definition of $g = f \convZ \varphi$,
there exist $y_{\varepsilon}$ and $z_{\varepsilon}$ in $\ZZ\sp{n}$ for which
$f(y_{\varepsilon}) + \varphi(z_{\varepsilon}) \leq g(x) + \varepsilon$.
Let $k_{0}=\max ( \| y_{\varepsilon} \|_{\infty}, \| z_{\varepsilon} \|_{\infty} )$.
For any $k \geq k_{0}$, we have
$g_{k}(x) \leq f(y_{\varepsilon}) + \varphi(z_{\varepsilon})$.
Therefore,
$g_{k}(x) \leq g(x) + \varepsilon$.
This shows that $\lim_{k \to \infty} g_{k}(x) = g(x)$.
\end{proof}

\begin{lemma} \label{LMicflimit}
If a sequence of integrally convex functions
$g_{k}: \mathbb{Z}^{n} \to \mathbb{R} \cup \{ +\infty  \}$
converges pointwise to a function 
$g: \mathbb{Z}^{n} \to \mathbb{R} \cup \{ +\infty  \}$
with an integrally convex effective domain,
then $g$ is also integrally convex.
\end{lemma}
\begin{proof}
By Theorem~\ref{THfavtarProp33} it suffices to show
\begin{equation} \label{floccv}
\tilde{g}(u) 
\leq \frac{1}{2} (g(x\sp{(1)}) + g(x\sp{(2)}))
\end{equation} 
for any $x\sp{(1)}, x\sp{(2)} \in \dom g$ with $\| x\sp{(1)} - x\sp{(2)} \|_{\infty} =2$,
where $u = (x\sp{(1)} + x\sp{(2)})/2$.
For each $k$ we have
\begin{equation} \label{fkloccv}
\tilde{g}_{k} ( u ) 
\leq \frac{1}{2} (g_{k}(x\sp{(1)}) + g_{k}(x\sp{(2)}))
\end{equation} 
by Theorem~\ref{THfavtarProp33}.
Recall that the local convex extension $\tilde{g}_{k}(u)$ is defined as
\begin{equation*}
 \tilde g_{k}(u) = 
  \min\{ \sum_{y \in N( u )} \lambda_{y} g_{k}(y) \mid
      \sum_{y \in N( u )} \lambda_{y} y = u ,  \ 
  (\lambda_{y})  \in \Lambda( u ) \} ,
\end{equation*} 
where the integral neighborhood $N(u)$ is independent of $k$.
For each simplex $\Delta$ with vertices taken from $N(u)$,
consider the linear interpolation of $g_{k}$ on $\Delta$ 
and denote its value at $u$ 
by
$\tilde g_{k}\sp{\Delta}(u)$,
where 
$\tilde g_{k}\sp{\Delta}(u) = +\infty$
if $u \not\in \Delta$.
Then 
$\tilde g_{k}(u)$ is equal to the minimum of 
$\tilde g_{k}\sp{\Delta}(u)$
over all $\Delta$, which are finite in number.
By defining $\tilde g\sp{\Delta}(u)$ from $g$ in a similar manner 
we have
$\displaystyle 
\tilde g(u) = \min_{\Delta} \tilde g\sp{\Delta}(u)$.
The pointwise convergence of 
$g_{k}$ to $g$ implies
$\displaystyle 
\lim_{k \to \infty} \tilde g_{k}\sp{\Delta}( u ) 
=  \tilde g\sp{\Delta}(u)$
for each $\Delta$, and hence
\[
 \lim_{k \to \infty} \tilde g_{k}( u ) 
 = \lim_{k \to \infty} \min_{\Delta} \tilde g_{k}\sp{\Delta}(u)
 = \min_{\Delta} \lim_{k \to \infty}  \tilde g_{k}\sp{\Delta}(u)
 = \min_{\Delta} \tilde g\sp{\Delta}(u)
 = \tilde g( u ).
\]
Therefore, (\ref{floccv}) follows from (\ref{fkloccv})
by letting $k \to \infty$.
\end{proof}

\subsection{Proof of Theorem~\ref{THlsepconvol} by discrete midpoint convexity}
\label{SClnatsepconvol}

In this section we prove the following proposition.
This serves as an alternative proof for Theorem~\ref{THlsepconvol}.

\begin{proposition} \label{PRmidptsepconvol}
If $f$ satisfies 
discrete midpoint convexity {\rm (\ref{lnatfmidconv})} for all $x,y \in \ZZ^{n}$,
so does its convolution $f \convZ \varphi$
with a separable convex function $\varphi$.
\end{proposition}
\begin{proof}
Let
$x\sp{(1)},x\sp{(2)} \in \dom g$,
and take any 
$\varepsilon > 0$.
There exist
$y\sp{(i)},z\sp{(i)}$ $(i=1,2)$ such that
\begin{align}  
g(x\sp{(i)}) & \geq f(y\sp{(i)}) + \varphi(z\sp{(i)}) - \varepsilon  ,
\quad
x\sp{(i)}= y\sp{(i)}+z\sp{(i)}
\qquad (i=1,2) .
\label{lnatprf1and2} 
\end{align}
By discrete midpoint convexity (\ref{lnatfmidconv}) of $f$ we have
\begin{equation}  \label{lnatprf3} 
 f(y\sp{(1)}) + f(y\sp{(2)}) \geq
    f \left(\left\lceil 
          \frac{y\sp{(1)}+y\sp{(2)}}{2} 
    \right\rceil\right) 
  + f \left(\left\lfloor 
          \frac{y\sp{(1)}+y\sp{(2)}}{2} 
     \right\rfloor\right) .
\end{equation}
Define
\begin{equation}  \label{lnatprf4} 
 z' = 
\left\lceil 
     \frac{x\sp{(1)}+x\sp{(2)}}{2} 
\right\rceil
- \left\lceil 
     \frac{y\sp{(1)}+y\sp{(2)}}{2} 
\right\rceil ,
\quad
 z'' = 
\left\lfloor
     \frac{x\sp{(1)}+x\sp{(2)}}{2} 
\right\rfloor
- \left\lfloor 
     \frac{y\sp{(1)}+y\sp{(2)}}{2} 
\right\rfloor .
\end{equation}
By the definition of $f \convZ \varphi = g$ we have
\begin{align} 
 \label{lnatprf6u} 
    f \left(\left\lceil 
          \frac{y\sp{(1)}+y\sp{(2)}}{2} 
    \right\rceil\right) 
  + \varphi(z')
 &\geq
    g \left(\left\lceil 
          \frac{x\sp{(1)}+x\sp{(2)}}{2} 
    \right\rceil\right) ,
\\
    f \left(\left\lfloor 
          \frac{y\sp{(1)}+y\sp{(2)}}{2} 
    \right\rfloor\right) 
  + \varphi(z'')
 &\geq
    g \left(\left\lfloor 
          \frac{x\sp{(1)}+x\sp{(2)}}{2} 
    \right\rfloor\right) .
 \label{lnatprf6d} 
\end{align}
For $\varphi(z) = \sum_{i=1}\sp{n} \varphi_{i}(z_{i})$,
on the other hand, we can show (see below) 
\begin{equation}  \label{lnatprf11} 
\varphi_{i}(z\sp{(1)}_{i}) + \varphi_{i}(z\sp{(2)}_{i})
\geq \varphi_{i}(z'_{i})  + \varphi_{i}(z''_{i})
\end{equation}
for $i=1,2,\ldots, n$, from which follows
\begin{equation}  \label{lnatprf5} 
\varphi(z\sp{(1)}) + \varphi(z\sp{(2)})
\geq \varphi(z')  + \varphi(z'') .
\end{equation}
By adding 
(\ref{lnatprf1and2}),
(\ref{lnatprf3}),
(\ref{lnatprf6u}),
(\ref{lnatprf6d}),
and
(\ref{lnatprf5}),
we obtain
\begin{equation}  \label{lnatprf7} 
 g(x\sp{(1)}) + g(x\sp{(2)}) \geq
    g \left(\left\lceil 
          \frac{x\sp{(1)}+x\sp{(2)}}{2} 
    \right\rceil\right) 
  + g \left(\left\lfloor 
          \frac{x\sp{(1)}+x\sp{(2)}}{2} 
     \right\rfloor\right) 
 - 2 \varepsilon.
\end{equation}
This implies discrete midpoint convexity (\ref{lnatfmidconv}) for $g$, 
since $\varepsilon > 0$ is arbitrary.

It remains to prove (\ref{lnatprf11}).
First we note a simple consequence of the convexity of $\varphi_{i}$.
Let $a$ and $b$ be integers with $a \leq b$,
and $p, q \in \ZZ$.
 If 
(i) $a + b = p+q$, 
(ii) $a \leq p \leq b$, and 
(iii) $a \leq q \leq b$, then
$\varphi_{i}(a) + \varphi_{i}(b) \geq \varphi_{i}(p) + \varphi_{i}(q)$.
Therefore the proof of (\ref{lnatprf11}) is reduced to showing the following:
\begin{align}  
&z\sp{(1)}_{i} + z\sp{(2)}_{i} = z'_{i} + z''_{i} ,
\label{lnatprf8} 
\\
&\min (z\sp{(1)}_{i},z\sp{(2)}_{i})
\leq z'_{i} \leq
\max (z\sp{(1)}_{i},z\sp{(2)}_{i}),
\label{lnatprf9} 
\\ 
&\min (z\sp{(1)}_{i},z\sp{(2)}_{i})
\leq z''_{i} \leq
\max (z\sp{(1)}_{i},z\sp{(2)}_{i}) .
\label{lnatprf10} 
\end{align}
The first equation (\ref{lnatprf8})
is a consequence of the identity
$\lceil \xi/2 \rceil + \lfloor  \xi/2 \rfloor = \xi$
valid for any $\xi \in \ZZ$ as follows:
\begin{align*}  
 z_{i}' + z_{i}'' 
&= 
\left(
\left\lceil 
     \frac{x_{i}\sp{(1)}+x_{i}\sp{(2)}}{2} 
\right\rceil
- \left\lceil 
     \frac{y_{i}\sp{(1)}+y_{i}\sp{(2)}}{2} 
\right\rceil 
\right)
+ 
\left(
\left\lfloor
     \frac{x_{i}\sp{(1)}+x_{i}\sp{(2)}}{2} 
\right\rfloor
- \left\lfloor 
     \frac{y_{i}\sp{(1)}+y_{i}\sp{(2)}}{2} 
\right\rfloor
\right)
\\ &=
\left(
\left\lceil 
     \frac{x_{i}\sp{(1)}+x_{i}\sp{(2)}}{2} 
\right\rceil
+\left\lfloor
     \frac{x_{i}\sp{(1)}+x_{i}\sp{(2)}}{2} 
\right\rfloor
\right)
- 
\left(
\left\lceil 
     \frac{y_{i}\sp{(1)}+y_{i}\sp{(2)}}{2} 
\right\rceil 
+ 
\left\lfloor 
     \frac{y_{i}\sp{(1)}+y_{i}\sp{(2)}}{2} 
\right\rfloor
\right)
\\ &=
(x_{i}\sp{(1)}+x_{i}\sp{(2)}) - (y_{i}\sp{(1)}+y_{i}\sp{(2)})
=
(x_{i}\sp{(1)}-y_{i}\sp{(1)}) + (x_{i}\sp{(2)}-y_{i}\sp{(2)})
=
z_{i}\sp{(1)} + z_{i}\sp{(2)} .
\end{align*}
To show (\ref{lnatprf9}) and (\ref{lnatprf10})
we substitute
$z_{i}\sp{(1)}+z_{i}\sp{(2)} = (x_{i}\sp{(1)}+x_{i}\sp{(2)}) - (y_{i}\sp{(1)}+y_{i}\sp{(2)})$
into
\[
 \min( z_{i}\sp{(1)}, z_{i}\sp{(2)})  \leq  \frac{z_{i}\sp{(1)}+z_{i}\sp{(2)}}{2} \leq  
  \max( z_{i}\sp{(1)}, z_{i}\sp{(2)}),
\]
to obtain
\[
 \min( z_{i}\sp{(1)}, z_{i}\sp{(2)}) +  
     \frac{y_{i}\sp{(1)}+y_{i}\sp{(2)}}{2} 
\leq  
     \frac{x_{i}\sp{(1)}+x_{i}\sp{(2)}}{2} 
 \leq  
 \max( z_{i}\sp{(1)}, z_{i}\sp{(2)}) +  
\frac{y_{i}\sp{(1)}+y_{i}\sp{(2)}}{2} .
\]
By applying 
$\left\lceil \ \cdot \ \right\rceil$
and
$\left\lfloor  \ \cdot \ \right\rfloor$,
we obtain
\begin{align*} 
& \min( z_{i}\sp{(1)}, z_{i}\sp{(2)}) +  
\left\lceil 
     \frac{y_{i}\sp{(1)}+y_{i}\sp{(2)}}{2} 
\right\rceil
\leq  
\left\lceil 
     \frac{x_{i}\sp{(1)}+x_{i}\sp{(2)}}{2} 
\right\rceil
 \leq  
 \max( z_{i}\sp{(1)}, z_{i}\sp{(2)}) +  
\left\lceil 
     \frac{y_{i}\sp{(1)}+y_{i}\sp{(2)}}{2} 
\right\rceil ,
\\ &
 \min( z_{i}\sp{(1)}, z_{i}\sp{(2)}) +  
\left\lfloor 
     \frac{y_{i}\sp{(1)}+y_{i}\sp{(2)}}{2} 
\right\rfloor
\leq  
\left\lfloor 
     \frac{x_{i}\sp{(1)}+x_{i}\sp{(2)}}{2} 
\right\rfloor
 \leq  
 \max( z_{i}\sp{(1)}, z_{i}\sp{(2)}) +  
\left\lfloor 
     \frac{y_{i}\sp{(1)}+y_{i}\sp{(2)}}{2} 
\right\rfloor ,
\end{align*} 
which are equivalent
to (\ref{lnatprf9}) and  (\ref{lnatprf10}), respectively.
\end{proof}

\section{Concluding Remarks}
\label{SCconclrem}

Besides projection and convolution, there are a number of fundamental
operations for discrete convex functions.
Here we touch upon conjugation,
restriction, and addition operations
for integrally convex functions.

For a function $f: \ZZ\sp{n} \to \RR \cup \{ +\infty \}$
with $\dom f \not= \emptyset$,
the (integer) conjugate of $f$
is the function 
$ f\sp{\bullet}: \ZZ\sp{n} \to \RR \cup \{ +\infty \}$
defined by 
\begin{equation} \label{conjvex}
 f\sp{\bullet}(p) 
 = \sup\{  \sum_{i=1}^{n} p_{i} x_{i} - f(x)     \mid x \in \ZZ\sp{n} \}
\qquad ( p \in \ZZ\sp{n}),
\end{equation}
which is a discrete version of the Fenchel--Legendre transformation.
The conjugate of an integrally convex 
(resp., globally or locally discrete midpoint convex) function
$f$  is not necessarily
integrally convex
(resp., globally or locally discrete midpoint convex).
This is shown by the following example.

\begin{example}[{\cite[Example 4.15]{MS01rel}}] \rm \label{EXconj-icf}
$ S = \{(1, 1, 0, 0), (0, 1, 1, 0), (1, 0, 1, 0), (0, 0, 0, 1)\} $ 
is obviously an integrally convex set,
as it is a subset of $\{ 0,1 \}\sp{4}$.
Accordingly, its indicator function $\delta_S: \ZZ^4 \to \{0, + \infty\}$ 
is integrally convex.
 The conjugate function $g = \delta_S^\bullet$ is given by
\[
 g(p_1, p_2, p_3, p_4) 
 = \max\{p_1 + p_2, p_2 + p_3, p_1 + p_3, p_4\} \qquad (p \in \ZZ^4).
\]
For $p =(0,0,0,0)$ and $q=(1,1,1,2)$ 
we have 
$\tilde g ((p+q)/2) > (g(p)+ g(q)) /2$,
a violation of the inequality (\ref{intcnvconddist2})
in Theorem~\ref{THfavtarProp33},
where $\tilde g$ denotes the local convex closure of $g$.
Indeed, 
$(p+q)/2 =(1/2,1/2,1/2,1)$ and 
$\tilde g ((p+q)/2) = 3/2$,
whereas $(g(p)+ g(q)) /2 = (0+2)/2 = 1$.
Hence $g$ is not integrally convex.
Moreover,  $S$ is a discrete midpoint convex set,
and the indicator function $\delta_S$ 
is globally (and hence locally) discrete midpoint convex.
Its conjugate function $g$ is not globally or locally discrete midpoint convex,
as it is not integrally convex.
\finbox
\end{example}

As is well known in convex analysis,
the operations that are conjugate to projection and convolution
are restriction and addition, respectively.
For a function
$f: \mathbb{Z}\sp{n+m} \to \mathbb{R} \cup \{ +\infty  \}$,
the restriction of $f$ to $\mathbb{Z}\sp{n}$
is the function
$g: \mathbb{Z}\sp{n} \to \mathbb{R} \cup \{ +\infty  \}$
defined by 
\begin{equation} 
 g(x) = f(x,\veczero_{m})
  \qquad (x \in \ZZ\sp{n}) ,
\label{fsetrestrict}
\end{equation}
where $\veczero_{m}$ means the zero vector in $\ZZ\sp{m}$.
It is easy to see that the restriction of an integrally convex 
(resp., globally or locally discrete midpoint convex) function
$f$  is integrally convex
(resp., globally or locally discrete midpoint convex).

The sum of integrally convex functions
is not necessarily integrally convex,
as the following example shows.
On the other hand, it is known \cite{MMTT16proxICissac,MMTT17midpt}
(and easy to see)
that the sum of globally (resp.  locally) discrete midpoint convex functions
is globally (resp.  locally) discrete midpoint convex.

\begin{example}[{\cite[Example 4.4]{MS01rel}}] \rm \label{ex:cap-ics}
Consider
$D_1  =  \{(0, 0, 0), (0, 1, 1), (1, 1, 0), (1, 2, 1)\}$
and
$D_2  =  \{(0, 0, 0), (0, 1, 0), (1, 1, 1), (1, 2, 1)\}$,
which are both integrally convex sets.
Their intersection $D_1 \cap D_2= \{(0, 0, 0), (1, 2, 1)\}$
is not an integrally convex set.
Therefore, the indicator functions 
$\delta_{D_1}, \delta_{D_2} : \ZZ^3 \to \{0, + \infty\}$ 
are integrally convex, and their sum
$\delta_{D_1} + \delta_{D_2}$
is not integrally convex.
\finbox
\end{example}

The results of this paper for convolution and projection operations 
are summarized in Table~\ref{TBconvolproj}
together with the previously known results for other discrete convex functions.

\begin{table}
\caption{Convolution and projection operations for discrete convex functions ($\varphi$: separable convex)}
\label{TBconvolproj}
\begin{center}
\begin{tabular}{c|ccc}
\hline
function class & convolution &  convol.~with sep. conv.   & projection
\\
$\mathcal{C}$ & $\mathcal{C} \convZ \mathcal{C} \subseteq \mathcal{C}$ &  
$\mathcal{C} \convZ \varphi \subseteq \mathcal{C}$ &   
${\rm proj}(\mathcal{C}) \subseteq \mathcal{C}$ 
\\
\hline
\hline
separable  
 & true & true &  true 
\\	
 convex  & (obvious) & (obvious) & (obvious) 
\\   \hline
${\rm M}^{\natural}$-convex & true & true & true 
\\   &
\cite[Th.~6.15]{Mdcasiam} & \cite[Th.~6.15]{Mdcasiam} & \cite[Th.~6.15]{Mdcasiam}
\\   \hline
${\rm L}^{\natural}$-convex & false & true & true 
\\   &
Ex.\ref{EXlnatsetsum}  & \cite[Th.~7.11]{Mdcasiam} & \cite[Th.~7.11]{Mdcasiam}
\\   \hline
integrally & false & true & true 
\\  
convex &
  Ex.\ref{EXicdim2sumhole}   &  Th.~\ref{THicsetbox}  & Th.~\ref{THicvsetproj} 
\\  \hline
globally discrete & false & false & true 
\\  
midpoint convex &
Ex.\ref{EXicdim2sumhole}  &  Ex.\ref{EXdicdim3set}, Ex.\ref{EXdicdim3fn}
 & Th.~\ref{THicvproj}
\\  \hline 
locally discrete & false & false & true 
\\   
midpoint convex & Ex.\ref{EXicdim2sumhole} & 
 Ex.\ref{EXdicdim3set}, Ex.\ref{EXdicdim3fn}
& Th.~\ref{THsdicvproj} 
\\  
\hline
\end{tabular}  
\end{center}
\end{table}

\paragraph{Acknowledgement}
The authors thank Fabio Tardella for helpful comments, especially 
for suggesting the use of projection operation described in Remark \ref{RMuseproj}.


\newpage

\begin{thebibliography}{99}


\bibitem{FT90}
Favati, P., Tardella, F.:
Convexity in nonlinear integer programming.
Ricerca Operativa {\bf 53}, 3--44 (1990)



\bibitem{Fuj14bisubmdc} 
Fujishige, S.: 
Bisubmodular polyhedra, simplicial divisions, and discrete convexity.
Discrete Optimization  {\bf 12}, 115--120 (2014)



\bibitem{FM00}
Fujishige, S., Murota, K.:
Notes on L-/M-convex functions and the separation theorems.
Mathematical Programming {\bf 88}, 129--146 (2000)



\bibitem{HL01}
Hiriart-Urruty, J.-B., Lemar{\'e}chal, C.:
Fundamentals of Convex Analysis.
Springer, Berlin (2001) 




\bibitem{Iim10} 
Iimura, T.:
Discrete modeling of economic equilibrium problems.
Pacific Journal of Optimization {\bf 6}, 57--64 (2010)




\bibitem{IMT05} 
Iimura, T., Murota, K., Tamura, A.:
Discrete fixed point theorem reconsidered.
Journal of Mathematical Economics {\bf 41}, 1030--1036 (2005)



\bibitem{IW14} 
Iimura, T., Watanabe,  T.:
Existence of a pure strategy equilibrium in finite symmetric games 
where payoff functions are integrally concave.
Discrete Applied Mathematics {\bf 166},  26--33 (2014)





\bibitem{LTY11nle} 
van der Laan, G., Talman, D., Yang, Z.:
Solving discrete systems of nonlinear equations.
European Journal of Operational Research {\bf 214}, 493--500 (2011) 



\bibitem{MMTT16proxICissac} 
Moriguchi, S., Murota, K.,  Tamura, A.,    Tardella, F.:
Scaling and proximity properties of integrally convex functions.
In: Seok-Hee Hong (ed.)  ISAAC2016,
Leibniz International Proceedings in Informatics (LIPIcs),
{\bf 64}, Article No. 57, 57:1--57:12 (2016)


\bibitem{MMTT17proxIC} 
Moriguchi, S., Murota, K.,  Tamura, A.,    Tardella, F.:
Scaling, proximity, and optimization of integrally convex functions.
To appear in Mathematical Programming,
arXiv 1703.10705, March 2017




\bibitem{MMTT17midpt} 
Moriguchi, S., Murota, K.,  Tamura, A.,    Tardella, F.:
Discrete midpoint convexity.
arXiv 1708.04579, August 2017



\bibitem{Mdca98} 
Murota, K.:
Discrete convex analysis. 
Mathematical Programming  {\bf 83}, 313--371 (1998)




\bibitem{Mdcasiam} 
Murota, K.:
Discrete Convex Analysis.
Society for Industrial and Applied Mathematics, Philadelphia (2003)



\bibitem{Mbonn09} 
Murota, K.:
Recent developments in discrete convex analysis.
In: Cook, W., Lov{\'a}sz, L., Vygen, J. (eds.)
Research Trends in Combinatorial Optimization,
Chapter 11, pp.~219--260. Springer, Berlin (2009) 



\bibitem{Mdcaeco16} 
Murota, K.:
Discrete convex analysis: A tool for economics and game theory.
{Journal of Mechanism and Institution Design} 
{\bf 1}, 151--273 (2016)



\bibitem{MS01rel} 
Murota, K., Shioura, A.:
Relationship of M-/L-convex functions with 
discrete convex functions by Miller and by Favati--Tardella.
Discrete Applied Mathematics {\bf 115}, 151--176 (2001)





\bibitem{Yan08comp}
Yang, Z.:
On the solutions of discrete nonlinear complementarity and related problems.
Mathematics of Operations Research {\bf  33}, 976--990 (2008)



\bibitem{Yan09fixpt} 
Yang, Z.:
Discrete fixed point analysis and its applications.
Journal of Fixed Point Theory and Applications {\bf 6}, 351--371 (2009)


\end{thebibliography}
\end{document}